\documentclass[11pt,a4paper]{amsart}
\usepackage{ucs}
\usepackage [utf8x]{inputenc}
\usepackage[english]{babel}
\usepackage {amsmath,amssymb,amstext, amsthm}
\usepackage{helvet}
\usepackage[numbers]{natbib}
\usepackage{graphicx}
\usepackage{tikz}

\usetikzlibrary{matrix}

\newcommand{\git}{\mathbin{\!
  \mathchoice{/\mkern-6mu/}% \displaystyle
    {/\mkern-5mu/}% \textstyle
    {/\mkern-5mu/}% \scriptstyle
    {/\mkern-5mu/}}\!}% \scriptscriptstyle

\newcommand{\R}{\mathbb{R}}
\newcommand{\C}{\mathbb{C}}

\newcommand{\g}{\mathfrak{g}}
\renewcommand{\k}{\mathfrak{k}}
\renewcommand{\O}{\mathcal{O}}
\newcommand{\F}{\mathcal{F}}
\newcommand{\G}{\mathcal{G}}
\renewcommand{\L}{\mathcal{L}}

\newtheorem{theorem}{Theorem}[section]
\newtheorem*{theoremGlobalExtension}{Theorem \ref{GlobalExtension}}

\newtheorem*{theoremPE}{Theorem \ref{PseudoconvexEmbedding}}
\newtheorem*{theoremLBE}{Corollary \ref{LineBundleEmbedding}}

\newtheorem{lemma}[theorem]{Lemma}
\newtheorem{proposition}[theorem]{Proposition}
\newtheorem{corollary}[theorem]{Corollary}

\usepackage[a4paper,top=3cm,bottom=3cm,left=3cm,right=3cm]{geometry}

\newenvironment{remark}[1][Remark]{\begin{trivlist}
\item[\hskip \labelsep {\bfseries #1}]}
{\end{trivlist}}
\newenvironment{definition}[1][Definition]{\begin{trivlist}
\item[\hskip \labelsep {\bfseries #1}]}
{\end{trivlist}}

\newenvironment{example}[1][Example]{\begin{trivlist}
\item[\hskip \labelsep {\bfseries #1}]}
{\end{trivlist}}

\begin{document}

\title{Equivariant embeddings of strongly pseudoconvex Cauchy-Riemann manifolds}
\author{Kevin Fritsch and Peter Heinzner}
\thanks{Kevin Fritsch and Peter Heinzner were supported by the CRC TRR 191: “Symplectic Structures in Geometry, Algebra
and Dynamics”}

\begin{abstract}
Let $X$ be a CR manifold with transversal, proper CR $G$-action. 
We show that $X/G$ is a complex space such that the quotient map is a CR map.
Moreover the quotient is universal, i.e. every invariant CR map into a complex manifold factorises uniquely over a holomorphic map on $X/G$.
We then use this result and complex geometry to proof an embedding theorem for (non-compact) strongly pseudoconvex CR manifolds with transversal $G \rtimes S^1$-action. 
The methods of the proof are applied to obtain a projective embedding theorem for compact CR manifolds. 
\end{abstract}

\maketitle

\section*{Introduction}

An important and much studied question in CR geometry is whether an abstract CR manifold can be realised, locally or even globally, as a CR submanifold of $\C^n$, see for example \cite{Andreotti2}, \cite{Boutet} or \cite{Kohn}.
There have also been several works on the topic of CR manifolds with transversal group actions \cite{BRT}, Lempert proved an embedding result for the otherwise difficult $3$-dimensional case assuming the existence of a transversal CR $\R$-action \cite{Lempert}.

There have been more recent results for CR manifolds with transversal $S^1$-action by Herrmann, Hsiao and Li \cite{HsiaoHerrmann} and also an equivariant Kodaira embedding theorem by Hsiao, Li and Marinescu \cite{HsiaoMarinescu}.

Most of the above results are for the case of CR codimension $1$, but the high codimension case is also interesting, as the following example from the theory of transformation groups shows.

Let $(Z, \omega)$ be a Kähler manifold with holomorphic action of a Lie group $G$ which leaves $\omega$ invariant. 
Let $\mu \colon Z \rightarrow \g^*$ be a momentum map such that $0$ is a regular value and define $\mathcal{M} := \mu^{-1}(0)$.
Then $T_x \mathcal{M} \cap i T_x \mathcal{M} = (\C \g (x))^{\perp_\omega}$ and $\mathcal{M}$ is a CR submanifold of $Z$ with transversal $G$-action. Here, $\g  (x)$ denotes the tangent space at $x$ to the orbit $G x$ and $\C \g (x)$ the complex subspace generated by $\g (x)$.

The space $\mathcal{M}$ and the induced complex structure on the quotient $\mathcal{M} /G$ is of high interest in geometric invariant theory.
In the case where $Z$ is a bounded domain in $\C^n$ with Bergmann metric $\omega$ and $G$ a unipotent subgroup of the group of holomorphic isometries of $(Z, \omega)$, very little is generally known about the quotient $\mathcal{M} / G$.  

In this paper we will consider $X$ to be a CR manifold with proper, transversal action of a group $G$ such that $G$ is a subgroup of its universal complexification $G^\C$.
We recall that $G$ is a subgroup of its universal complexification if and only if it is a Lie subgroup of some complex Lie group. 
For example every closed subgroup of a matrix group is a subgroup of a complex group but the universal covering of $SL_2 ( \R)$ is not.  

We systematically start by showing that $X$ may always be embedded into a complex manifold. 
In particular, we say that a complex manifold $Z$ with holomorphic action of $G^\C$ and a $G$-equivariant CR embedding $\Phi \colon X \rightarrow Z$ is the universal equivariant extension of $X$ if every equivariant CR map $f \colon X \rightarrow Y$ into a complex manifold $Y$ with holomorphic $G^\C$-action extends uniquely to a $G^\C$-equivariant holomorphic map on $Z$.

\begin{theoremGlobalExtension}
Let $X$ be a CR manifold with proper, transversal, CR action of a Lie group $G$. Assume that $G$ is a subgroup of its universal complexification.
Then there exists a universal equivariant extension for $X$.
\end{theoremGlobalExtension}

We use this to show that the quotient space $X/G$ carries the structure of a complex space such that the sheaf of holomorphic functions on $X/G$ is given by the sheaf of $G$-invariant CR functions on $X$ (Theorem \ref{QuotientComplex}).
In section \ref{SectionPseudoconvex}, we will generalise the notion of strong pseudoconvexity to CR manifolds with transversal of codimension one group action.  
We will then proof a connection between strongly pseudoconvex CR manifolds and $S^1$-bundles in positive orbifold line bundles (see Theorem \ref{OrbifoldBundle}).

Using the quotient result and methods from complex geometry, we proof the following equivariant embedding theorem.
Let $H$ be a closed subgroup of its universal complexification, such that $H^\C$ is complex reductive and $H = G \rtimes S^1$, then we have $H^\C = G^\C \rtimes \C^*$. 
Let $X$ be a CR manifold with proper, transversal, CR action of $H$ such that $H^0_x < G_x^0$ for every $x \in X$ and let $Y$ be the universal equivariant exension of $X$.

\begin{theoremPE}
Under the assumptions above, let $X$ be strongly pseudoconvex and $X/G$ be compact.
Then there exists a $H^\C$-representation $V$ and a $H^\C$-equivariant holomorphic embedding $\Phi \colon Y \rightarrow \C^m \backslash \{0\} \times V$, such that $\Phi|_X \colon X \rightarrow \C^m \times V$ is an embedding. Here, $\C^m$ is the trivial $G^\C$-representation and decomposes into irreducible $\C^*$-representations with positive weights.
\end{theoremPE}

In chapter \ref{ProjectiveEmbedding}, we use these techniques for a proof of a Kodaria-type embedding theorem for CR manifolds.

\begin{theoremLBE}
Let $X$ be a compact CR manifold with a transversal CR action of a compact Lie group $K$.
Assume that there exists a weakly negative line bundle $L_K \rightarrow X/K$ and let $L \rightarrow X$ be the induced CR line bundle.
Then there exists a natural number $k$ and finitely many CR sections $s_i \in \Gamma (X, L^{-k})$ such that, for $W = \textrm{span}(s_i)$, we have that
\begin{gather*}
X  \rightarrow \mathbb{P}(W^*) \\
y \mapsto [s \mapsto s(y)]
\end{gather*}
is a CR embedding.
\end{theoremLBE}

Moreover we will show that the above embedding can be chosen to be $K$-equivariant. 

\section{Quotients}\label{SectionPreliminaries}

Let $X$ be a smooth manifold and $G$ a Lie group with smooth action on $X$.
For an element $\xi \in \g$, we denote by $\xi_X (x) := \frac{d}{dt}\big|_0 \textrm{exp}( t \xi) x$ the fundamental vector field of $\xi$ on $X$.
We set $\g (x) := \{ \xi_X (x) \, | \, \xi \in \g \}$ and say that the action is locally free if $\xi_X (x) \neq 0$ for every $\xi \in \g \backslash \{0\}$.\\
For a manifold $X$, we write $\C TX$ for the complexified tangent bundle. \\
Let $T^{1,0} X$ be a smooth complex subbundle of $\C TX$ such that $T^{1,0} X \cap \overline{T^{1,0} X } = \{0\}$ and $[\Gamma ( U, T^{1,0} X), \Gamma ( U , T^{1,0} X) ] \subset \Gamma (U,T^{1,0} X)$ for every open subset $U$ of $X$. 
Here, $\Gamma (U, T^{1,0} X)$ denote the smooth sections into $T^{1,0} X$ on $U$.
Set $n = $dim$_\C T^{1,0} X$ and $d = $dim$_\R X - 2n$, then we call $(X, T^{1,0} X)$ a CR manifold of dimension $(2n,d)$.
We write $T^{0,1}X := \overline{T^{1,0} X}$.\\
A typical example is given by a real submanifold $X$ of a complex submanifold $Z$ such that dim$(T_x X \cap J T_x X)$ is constant in $x$, where $J$ denotes the complex structure on $Z$.\\
Let $(X, T^{1,0}X)$ be a CR manifold with a CR action of a Lie group $G$. 
We call the action transversal, if $\C \g (x) \oplus T^{1,0}_x X \oplus T^{0,1}_x X = \C T_x X$ for every $x \in X$, where $\C \g (x)$ is the complex subspace in $\C T_x X$ generated by $\g (x)$.\\

A result of Loose states that if the action is proper, free and fulfils $\C \g (x) \cap T^{1,0}_x X \oplus T^{0,1}_x X = \{0\}$, then the quotient $X/G$ may be equipped with a CR structure such that the projection $X \rightarrow X/G$ is a CR map  \cite[Theorem 1.1]{LooseCR}.
Our first goal is to show that if the action is proper and transversal, but not necessarily free, then the quotient $X/G$ is a complex space.\\

For the proof, we first recall some basic techniques regarding quotients in the smooth case.

Let $X$ be a smooth manifold and $G$ a Lie group with proper and free action on $X$.
For every point $x \in X$, we find a smooth submanifold $S$ of $X$ with $x \in S$ such that the map $G \times S \rightarrow X$, $(g,s) \mapsto gs$ is a diffeomorphism onto an open subset $U$ of $X$ \cite[Theorem 2.3.3]{Palaisslice}.
The quotient map $\pi \colon X \rightarrow X/G$ induces a homeomorphism $\pi|_S \colon S \rightarrow \pi(U)$, which defined on $X/G$ the structure of a smooth manifold.

If $G$ is a closed subgroup of a Lie group $H$ and $G$ acts on a manifold $X$, then $G$ acts proper and free on $H \times X$ via $(g,(h,x)) \mapsto (hg^{-1}, gx)$, hence $H \times^G X := (H \times X) /G$ is a manifold.

This lets us formulate the general slice theorem as follows.
Let $G$ act on $X$ properly, then around every $x \in X$, there exists a slice, i.e. a smooth, $G_x$-invariant submanifold $S$ of $X$ such that the map $G \times^{G_x} S \rightarrow X$, $[g,s] \mapsto gs$ is a diffeomorphism onto an open subset of $X$ \cite[Theorem 2.3.3]{Palaisslice}.\\

On CR manifolds, it is in general not possible to construct a reasonable CR structure on the quotient using slices directly.
More precisely, if $G$ acts freely and transversally on a CR manifold $X$ and $S$ is a slice at $x \in X$ such that $\C T_x S = T^{1,0}_x X \oplus T^{0,1}_x X$, then this would imply that $T^{1,0} X \oplus T^{0,1} X$ is involutive and $X$ is flat. This fails in general.\\

We will use the following Lemma to construct complex structures on quotient spaces.
For a smooth map $f \colon X \rightarrow Y$ between manifolds, we may extend $d f$ to a $\C$-linear map $d f \colon \C T X \rightarrow \C TY$. 
For the sake of simplicity, we will denote the extension with $df$, as well.

\begin{lemma}\label{RelatedSections}
Let $X,Y$ be smooth manifolds and $\pi \colon X \rightarrow Y$ a surjective submersion. 
Let $E$ be a smooth, complex subbundle of $\C TX$.
For every $y \in Y$, let $F_y$ be a complex subspace of $\C T_y Y$ such that $d_x \pi \colon E_x \rightarrow F_{\pi(x)}$ is an isomorphism for every $x \in X$.
Then $F = \bigcup_y F_y$ is a smooth subbundle of $\C TY$ and for every $x \in X$, there exist open neighbourhoods $U$ of $x$ and $\Omega$ of $\pi(x)$ with $\pi(U) = \Omega$, such that for every smooth section $W \in \Gamma( \Omega, F)$, there exists a smooth, $\pi$-related section $V \in \Gamma( U,E)$, i.e. we have $d \pi \circ V = W \circ \pi$.
\end{lemma}
\begin{proof}
Denote by $B^n$ the ball with radius $1$ in $\R^n$.

Since the result is local in $X$, we may assume that $X = B^n \times B^d$, $Y = B^n$ and $\pi$ is the projection onto the first component.
We may also assume that we find smooth sections $V_i \in \Gamma( X, E)$, $i = 1,..,k$, such that $V_i(p)$ form a complex basis for $E_p$ in every point $p \in X$.

Now fix $w \in B^d$, set $s_w \colon B^n \rightarrow B^n \times B^d$, $z \mapsto (z,w)$ and define $W_i^w (z) := d_{(z,w)} \pi ( V_i( s_w(z)) )$.
Since $d \pi \colon E \rightarrow F$ is an isomorphism in every point, the $W_i^w(z)$ define a basis for $F_w$ for every $w \in B^d$ and $z \in B^n$ which depends smoothly on $w$ and $z$.
In particular, the space $F = \bigcup_w F_w$ defines a smooth subbundle of $\C TY$.\\

Now let $W \in \Gamma (Y,F)$ be a smooth section. 
For $w \in B^d$, we write $W (z) = \sum_i f_i^w (z) W_i^w (z)$ for complex-valued smooth functions $f_i^w$. 

Then the $\pi$-related smooth section is given by $V(z,w) := \sum_i f_i^w (z) V_i (z,w)$.
\end{proof}

We formalise the notion of CR embeddings.

\begin{definition}
Let $X$, $Y$ be CR manifolds. A map $\Phi \colon X \rightarrow Y$ is called a \textbf{CR embedding} if it is a smooth embedding of $X$ into $Y$ and
\begin{gather*}
d \Phi( T^{1,0}X) = d \Phi(\C TX) \cap T^{1,0} Y,
\end{gather*}
 i.e. $\Phi(X)$ is a CR submanifold of $Y$.
We require all embeddings to be \textbf{closed}.
\end{definition}

Note that in the case where $X$ is of CR codimension $1$, it suffices that $\Phi$ is an embedding and a CR map.
For higher CR codimension however, this is not true. \\

If a CR manifold $X$ may be embedded into a complex manifold $Z$, it is an interesting question whether the CR functions on $X$ may (locally) be extended to holomorphic functions on $Z$.

If $H$ is a complex Lie group acting on a complex manifold $Z$, we say that the action is holomorphic if the action map $H \times Z \rightarrow Z$ is holomorphic.

\begin{definition}
Let $G$ be a Lie group and $X$ a CR manifold with transversal action of $G$. 
Let $Z$ be a complex manifold with holomorphic $G^\C$-action and $\Phi \colon X \rightarrow Z$ a $G$-equivariant CR map.
We say that $(Z, \Phi)$ is a universal equivariant extension of $X$ if for every complex manifold $Y$ with holomorphic $G^\C$-action and every $G$-equivariant CR map $f \colon X \rightarrow Y$, there exists a unique $G^\C$-equivariant holomorphic map $F \colon Z \rightarrow Y$ with $f = F \circ \Phi$.
\end{definition}

If the universal equivariant extension exists, it is unique up to a $G^\C$-equivariant biholomorphic map.\\

From now on, denote by $G^\C$ the universal complexification of $G$ and assume that $G$ is a subgroup of $G^\C$. 
This is for example the case if there exists any injective morphism of $G$ into a complex Lie group.
In particular, it is the case for every linear group.
Now $G$ being a subgroup of $G^\C$ implies that $G$ is a totally real, closed subgroup of $G^\C$ \cite[§1 Proposition]{HCom}.

Let $X$ be a CR manifold with proper, transversal CR action of $G$.
For a Lie group $G$, we will denote by $G^0$ the connected component of the identity, which is a normal Lie subgroup of $G$.

Let $S$ be a (smooth) slice at $x \in X$ and $L = G_x$. Because all $G$-orbits are of the same dimension, all isotropy groups $G_s$ for $s \in S$ have to contain $L^0$.
Since $L^0$ acts trivially on $S$, we get an $L/L^0$-action on $S$ and the $L$-orbits on $S$ are finite. 

Since $L$ is compact, we have $L^\C = L \textrm{exp} (i \mathfrak{l})$ and every connected component of $L^\C$ intersects $L$. 
We conclude that $L^\C / (L^0)^\C = L / L^0$ and get an $L^\C$-action on $S$ as a finite group.

Because of \cite[§3 Corollary 1]{HCom}, we see that $L^\C$ is a closed complex subgroup of $G^\C$.
For $\Omega = G \times^L S$, define the smooth manifold 
\begin{equation*}
\Omega^\C := G^\C  \times^{L^\C} S.
\end{equation*}

Note that for $s \in S$, we get $(L_s)^\C = (L^\C)_s$.
Hence for $g \in G$ and $[g,s] \in G^\C \times^{L^\C} S$, we conclude $(G^\C)_{[g,s]}= g (L^\C)_s g^{-1} = g(L_s)^\C g^{-1} =  (G_{[g,s]})^\C$.

We call $\Omega^\C$ the \textbf{extension of the slice $S$} or just a \textbf{slice extension around $x$}. \\

We want to show that $\Omega$ may be embedded into $\Omega^\C$ and start with the following Lemma.

\begin{lemma}\label{GeneralEmbedding}         %%Prüfe ob man die referenz von peter für lie-gruppen kriegt. Dann kann man den beweis sehr einfach führen. Wenn man weiß, dass $G/G_0$ eine umf von H/H_0$ kriegt man über die bündeleigenschaft sofort die aussage. siehe auch pontryagin
Let $G$ be a closed subgroup of $H$. 
Let $H_0$ be a closed subgroup of $H$, define the subgroup $G_0 := G \cap H_0$ and assume that the image of the map $G/G_0 \rightarrow H / H_0$ is closed.
Let $X$ be a smooth manifold with $H_0$-action, then the map $\Phi \colon G \times^{G_0} X \rightarrow H \times^{H_0} X$ is an embedding. 
\end{lemma}
\begin{proof}
By construction of $G_0$ and $H_0$, the map $G/G_0 \rightarrow H/H_0$ is an immersion.
In particular, it is a smooth embedding \cite[§2.13 Theorem]{Montgomery}.
Since the following diagram of bundle maps 
\begin{center}
\begin{tikzpicture}
  \matrix (m) [matrix of math nodes,row sep=3em,column sep=4em,minimum width=2em]
  {
     G \times^{G_0} X & H \times^{H_0} X \\
     G /{G_0}  & H /{H_0} \\ };
  \path[-stealth]
    (m-1-1) edge node [left] {} (m-2-1)
            edge node [above] {$\Phi$} (m-1-2)
    (m-1-2) edge node [right] {} (m-2-2)
    (m-2-1) edge node [above] {} (m-2-2);
\end{tikzpicture}
\end{center}
commutes, it follows that $\Phi$ is an embedding.
\end{proof}

\begin{lemma} \label{LemmaSmoothEmbedding}
The natural map $\Phi \colon G \times^L S \rightarrow G^\C \times^{L^\C} S$ is an embedding.
\end{lemma}
\begin{proof}
The idea of the proof is to apply Lemma \ref{GeneralEmbedding} for $H = G^\C$, $H_0 = L^\C$ and $H = L$.
Due to \cite[§ 3 Corollary 1]{HCom}, we have $L^\C \cap G = L$.
 It remains to show that the image of $G/L \rightarrow G^\C / L^\C$ is closed.
 
From \cite[§1 Proposition]{HCom} we get the existence of an involutive anti-holomorphic homomorphism $\Theta$ on $G^\C$ such that $G$ is the fixed point set of $\Theta$.

Let $\theta$ be the corresponding involution on $\g^\C$, then $\mathfrak{l}$ is fixed under $\theta$ and $\mathfrak{l}^\C$ is invariant.
Since $L$ is compact, we have $L^\C = L \textrm{exp} (i \l)$.

Now let $g_n$ be a sequence in $G$ such that the image of $g_n$ in $G^\C / L^\C$ converges.
We find a sequence $l_n \in L^\C$ such that $g_n l_n \rightarrow g \in G^\C$.

Write $l_n = k_n \textrm{exp} (P_n)$ with $k_n \in L$ and $P_n \in i \mathfrak{l}$. 
Then 
\begin{equation*}
\begin{split}
&l_n g_n^{-1} \Theta(l_n g_n^{-1})^{-1} = l_n \Theta(l_n)^{-1} 
= k_n \textrm{exp}(P_n)\Theta(k_n \textrm{exp}(P_n))^{-1} \\
= &k_n \textrm{exp}(P_n) \textrm{exp}(- \theta(P_n))k_n^{-1} = k_n \textrm{exp} (2P_n) k_n^{-1},
\end{split}
\end{equation*}
using the same computation as above.
Since $L$ is compact, we may assume that $k_n$ and therefore $\textrm{exp} (P_n)$ converges.
This implies that $l_n$ and therefore $g_n$ is a convergent sequence.
\end{proof}

In order to formulate our first result, we fix $x \in X$ and a Slice $S$ at $x$. 
We identify $\Omega = G \cdot S$ with $G \times^L S$ where $L = G_x$ and denote by $\Omega^\C = G^\C \times^{L^\C} S$ the corresponding slice extension.
We have the following

\begin{theorem}\leavevmode\label{LocalExtension} 
\begin{enumerate}
\item The slice extension $\Omega^\C$ is a complex manifold such that the natural $G^\C$-action on $\Omega^\C$ is holomorphic.
\item The $G$-equivariant map $\Phi \colon G \times^L S \rightarrow G^\C \times^{L^\C} S$ is a CR embedding.
\item For every $G$-invariant open subset $U \subset \Omega$, the restriction $\Phi|_U \colon U \rightarrow G^\C U$ is a CR embedding and $G^\C U$ is a universal complexification of $U$.
\end{enumerate}
\end{theorem}																					
\begin{proof}
We begin by showing that $\Omega^\C$ is a complex manifold.
The map
\begin{gather*}
\eta \colon G^\C \times \Omega \rightarrow \Omega^\C \\
(h,x) \mapsto h \Phi(x).
\end{gather*}
is a surjective submersion.
We denote by $[h,s]_C$ the elements of $G^\C \times^{L^\C} S$. 						

The fibre $\eta^{-1} ([1,s_0]_C)$ consists of the elements $(h, [g,s]) \in G^\C \times (G \times^L S)$ such that $[hg, s]_C = [1,s_0]_C$.
In particular, this means that there exists an $l \in L^\C$ such that $l s = s_0$. But since the $L^\C$-orbits on $S$ are equal to the $L$-orbits, we may choose $l$ to be in $L$.
We get
\begin{gather*}
\eta^{-1}([1,s_0]_C ) = \{ (h, [g,s_0]) \, | \, [hg, s_0]_C = [1,s_0]_C \}.
\end{gather*}
This condition is equivalent to the existence of an $l \in (G^\C )_{[1,s_0]}= (G_{[1,s_0]})^\C$ such that $hg l^{-1} = 1$, hence $h = l g^{-1}$.
We have shown that
\begin{gather*}
\eta^{-1}([1,s_0]_C) = \{ (lg^{-1}, [g,s_0]) \, | \, l \in (G_{[1,s_0]})^\C, \, g \in G \}.
\end{gather*}
Now let $h_0 \in G^\C$. Since $\eta$ is $G^\C$-equivariant, we get
\begin{gather*}
\eta^{-1}([h_0, s_0] ) = \{ (h_0 l g^{-1}, [g, s_0]) \, | \, l \in (G_{[1,s_0]})^\C, \, g \in G \}.
\end{gather*}
Since $\eta$ is a submersion, every fibre of $\eta$ is a submanifold with tangent space equal to the kernel of $d \eta$.
hence in a point $(h_0, y) \in G^\C \times \Omega$, we get 					
\begin{gather*}
\textrm{ker} \, d_{(h_0,y)} \eta = \{  ( d h_0 (\xi- \mu), \mu_\Omega (y)) \, | \, \xi \in (\g_y)^\C, \mu \in \g \},
\end{gather*}
where $(\g_y)^\C$ denotes the Lie algebra of the isotropy group $(G^\C)_y$ and $dh_0$ is the differential of the left translation by $h_0$.

Now we consider $G^\C \times \Omega$ as a CR manifold.
Let $W \in \textrm{ker} \, d_{(h_0,y)} \eta \cap T^{1,0} (G^\C \times \Omega) $. Since the $G$-action on $\Omega$ is transversal, we have $W = (d h_0  (\xi),0)$ for $\xi \in \C(\g_y)^\C$.
Define 
\begin{gather*}
K_{(h_0, y)} := \textrm{ker} \, d_{(h_0,y)} \eta \cap T^{1,0}_{(h_0,y)} (G^\C \times \Omega) = \{ (d h_0 (\xi), 0) \, | \, \xi \in T^{1,0}_1 (G_y)^\C \}
\end{gather*}
then
\begin{gather*}
\overline{K}_{(h_0,y)}  = \{ (d h_0 (\xi), 0) \, | \, \xi \in T^{0,1}_1 (G_y)^\C \}
\end{gather*}
and
\begin{equation*}
\textrm{ker} \, d_{(h_0,y)} \eta \cap (T^{1,0} (G^\C \times \Omega)  \oplus T^{0,1} (G^\C \times \Omega))  = K_{(h_0,y)} \oplus \overline{K}_{(h_0,y)}.
\end{equation*}
Since all isotropy groups of the $G$-action on $\Omega$ are of the same dimension, we see that the dimension of $K_y$ does not depend on $y \in G^\C \times \Omega$, hence it defines a complex subbundle of $T^{1,0} (G^\C \times \Omega)$. 
Let $E$ be a complex subbundle of $T^{1,0} (G^\C \times \Omega)$ such that $E \oplus K = T^{1,0} (G^\C \times \Omega)$.\\

For $x \in G^\C \times \Omega$, define $F_x := \{ d_x \eta (V) \, | \, V \in T^{1,0}_x (G^\C \times \Omega) \}$ and for $y \in \Omega^\C$, set $F_y := \bigcup_{\eta (x) = y} F_x$.
Note that $d \eta$ defines an isomorphism between $E_x$ and $F_x$ in every point.
We want to show that $F_y = F_x$ for every $x \in \eta^{-1}(y)$.

The map $\eta$ is invariant under the CR action of $G$ on $G^\C \times \Omega$ via $(g, (h,x)) \mapsto (h g^{-1}, gx)$, which shows that $F_{[h,gx]} = F_{[hg, x]}$ for $g \in G$.

Because of that, it suffices to show $F_{(h, [1,s])} = F_{(h_0, [1,s_0])}$ if $\eta(h,[1,s]) = \eta (h_0, [1,s_0])$.
But the latter implies that $[h,s]_C = [h_0,s_0]_C$, hence there exists an $l \in L^\C$ such that $ls = s_0$ and we may again choose $l$ to be in $L$.

We may therefore assume $s = s_0$ and $h h_0^{-1} \in (G_{[1,s]})^\C$. It then remains to prove $F_{(l,[1,s])} = F_{(1,[1,s])}$ for all $l \in (G_{[1,s]})^\C$.

We have $F_{(l,[1,s])} =F_{[1,l[1,s]]} =  F_{(1,[1,s])}$ for all $l \in G_{[1,s]}$.
Note that $[1,s] \in \Omega$ implies $(G_{[1,s]})^\C = (G^\C)_{[1,s]}$.

Because $\eta$ is equivariant, we conclude $F_{(l,[1,s])} =d l (F_{(1,[1,s])})$ and need to show that $d l (F_{(1,[1,s])}) = F_{(1,[1,s])}$ for $l \in G_{[1,s]}^\C$. 
But now $\C T_{\eta(1,[1,s])} \Omega^\C$ is a holomorphic $G_{[1,s]}^\C$-representation and $F_{(1,[1,s])}$ is a complex subspace which is invariant under $G_{[1,s]}$.
Therefore it is invariant under $G_{[1,s]}^\C$. \\ %%because of the universality condition  

We have shown that $d_x \eta$ induces an isomorphism between $E_x$ and $F_{\eta(x)} = F_x$ and from Lemma \ref{RelatedSections} we get that $F = \bigcup_x F_{\eta(x)}$ is a subbundle of $\C T \Omega^\C$. 
We claim that $F$ defines a CR structure on $\Omega^\C$. \\

Let $d_x \eta (V), d_x \eta (W) \in F_{\eta(x)}$ and assume $d_x \eta (V) = \overline{ d_x \eta (W)} = d_x \eta ( \overline{W})$.
We conclude $V- \overline{W} \in K_x \oplus \overline{K_x} $ and get $V - \overline{W} = V_0 - \overline{W}_0$ with $V_0 \in K_x$ and $\overline{W}_0 \in \overline{K}_x$.
But then $V_0 = V$, $W_0 = W$ and $d_x \eta (V) = d_x \eta ( \overline{W}) = 0$. 

For $V \in \Gamma^\infty ( \Omega^\C, F)$, we use Lemma \ref{RelatedSections} and find a smooth section $V_0 \in \Gamma^\infty (G^\C \times \Omega,E)$ with $d_x \eta V_0 = V(\pi(x))$.
Then
\begin{gather*}
V(\eta(x)) (f) = d f ( V (\eta(x))) = d f ( d \eta (V_0 (x))) = V_0 (x) (f \circ \eta),
\end{gather*}
or alternatively $(Vf) \circ \eta = V_0 (f \circ \eta)$.
For another $W \in \Gamma^\infty (\Omega^\C, F)$ and $W_0$ as above, this then implies 
\begin{equation*}
\begin{split}
[V,W] (\eta(x)) (f) &= V(W(f)) ( \eta (x)) - W(V(f)) ( \eta (x))\\
& = V_0( W_0( f \circ \eta)) (x) - W_0(V_0( f \circ \eta )) (x)\\
& = [V_0, W_0] (x) ( f \circ \eta)
= d \eta ( [V_0, W_0] (x)) (f).
\end{split}
\end{equation*}
Hence $[V,W] ( \eta(x)) = d \eta ([V_0, W_0](x))$ and $[V,W](\eta(x)) \in F_{\eta(x)}$.

We have shown that $\Omega^\C$ is a CR manifold.
Since its real codimension is zero, it is a complex manifold.

The $G^\C$-action on $\Omega^\C$ is holomorphic if the pulled back map $G^\C \times G^\C \times \Omega \rightarrow \Omega^\C$, $(g, (h,x)) \mapsto gh \Phi (x)$ is CR.
This is true because $\eta$ is a CR map by construction.\\

Lemma \ref{LemmaSmoothEmbedding} says that $\Phi$ is a smooth embedding, let us check that it is a CR embedding.
Note that $\Phi$ is a CR map by construction and $\Phi_0 \colon \Omega \rightarrow G^\C \times \Omega$, $x \mapsto (1,x)$ is a CR embedding with $\Phi = \eta \circ \Phi_0$.

Let $y \in \Omega$ and $W \in d_y \Phi ( \C T_y \Omega) \cap T^{1,0}_{\eta(1,y)} \Omega^\C$. 
By construction of the CR structure, $W = d_{(1,y)} \eta (W_0)$ for $W_0 \in T^{1,0}_{(1,y)} ( G^\C \times \Omega)$ and $W = d_{y} \Phi (V_0)$ for $V_0 \in \C T_y \Omega$.

Now $W_0 - d_y \Phi_0 (V_0) \in \textrm{ker} \, d_{(1,y)} \eta = \{  ( \xi- \mu, \mu_\Omega (y)) \, | \, \xi \in \C(\g^\C)_y, \mu \in \C\g \}$.
We find $\xi \in \C (\g^\C)_y$, $\mu \in \C \g$ such that $W_0 = ( \xi- \mu, \mu_\Omega (y) + d_y \Phi_0 (V_0))$.
But $W_0 \in T^{1,0}_{(1,y)} (G^\C \times \Omega)$ and since $\g$ is totally real in $\g^\C$, we get $\C \g \cap T^{1,0}_1 G^\C = \{0 \}$ and $\mu = 0$.
Because $\Phi_0$ is a CR embedding, this implies $V_0 \in T^{1,0}_y \Omega$ and shows that $\Phi$ is a CR embdding.\\

Now let us consider the universality condition. We will assume $U = \Omega$, the general case is analogous.
Let $Y$ be a complex manifold with holomorphic $G^\C$-action and $f \colon \Omega \rightarrow Y$ a $G$-equivariant CR map. 

We define $F \colon \Omega^\C \rightarrow Y$, $hx \mapsto hf(x)$ for $h \in G^\C$, $x \in \Omega$.
For this to be well-defined, we need to check that $h \in (G^\C)_x$ implies $h \in (G^\C)_{f(x)}$.

Since $x \in \Omega$, we have $(G^\C)_x  = (G_x)^\C$ and because the action of $G^\C$ on $Y$ is holomorphic, $G_x \subset (G^\C)_{f(x)}$ implies $(G^\C)_x = (G_x)^\C \subset (G^\C)_{f(x)}$.

We need to show that $F$ is holomorphic. Consider the commuting diagram
\begin{center}
\begin{tikzpicture}
  \matrix (m) [matrix of math nodes,row sep=3em,column sep=4em,minimum width=2em]
  {
     G^\C \times \Omega & G^\C \times Y \\
     \Omega^\C & Y \\};
  \path[-stealth]
    (m-1-1) edge node [right] {$\eta$} (m-2-1)
            edge node [above] {id $\times f$} (m-1-2)
    (m-2-1) edge node [below] {$F$} (m-2-2)
    (m-1-2) edge node [right] {$\varphi$} (m-2-2);
\end{tikzpicture}
\end{center}
Where $\varphi$ is the action map. Now $F$ is holomorphic if and only if $F \circ \eta$ is CR, which follows from the diagram.
\end{proof}

We want to give a global version of this local statement. 
For this, we consider the union over all slice extensions and identify the overlapping parts. 
We formalise this as follows.

Around every $x \in X$, there exists a slice extension $\Omega_x^\C$.
We may cover $X$ with countably many sets $\Omega_i$ with extensions $\Omega^\C_i$ for $i \in \mathbb{N}$.
 
For $i,j \in \mathbb{N}$ with $\Omega_i \cap \Omega_j \neq \emptyset$, define the open subset 
\begin{gather*}
\Omega^\C_{ij} := G^\C \cdot (\Omega_i \cap \Omega_j) \subset \Omega^\C_i.
\end{gather*}
Then the identity map $\Omega_i \cap \Omega_j \rightarrow \Omega_i \cap \Omega_j$ extends to a unique $G^\C$-equivariant holomorphic map $\varphi_{ji} \colon \Omega^\C_{ij} \mapsto \Omega^\C_{ji}$, using Theorem \ref{LocalExtension}.

Because of the uniqueness, we get $\varphi_{ii} = \textrm{id}_{\Omega^\C_i}$ and $\varphi_{kj} \circ \varphi_{ji} = \varphi_{ki}$ on the open subset $G^\C ( \Omega_i \cap \Omega_j \cap \Omega_k) \subset \Omega_i^\C$.
 This also implies $\varphi_{ij} = \varphi_{ji}^{-1}$.
 
Now define
\begin{gather*}
Z := \bigcup_{i \in \mathbb{N}} \Omega^\C_i \Big/ \sim,
\end{gather*} 
where $x \in \Omega^\C_i$ and $y \in \Omega^\C_j$ are equivalent if $\Omega_i \cap \Omega_j \neq \emptyset$ and $\varphi_{ji} (x) = y$.
Because of the remarks above, this does indeed define an equivalence relation.

Define $Z_0 := \bigcup_{i \in \mathbb{N}} \Omega^\C_i $ to be the disjoint union over the $\Omega^\C_i$, the quotient map $\pi \colon Z_0 \rightarrow Z$ and equip $Z$ with the quotient topology.

\begin{lemma}\label{Hausdorff}  
The space $Z$ is Hausdorff and second countable.
\end{lemma}
\begin{proof}
Fix $i \in \mathbb{N}$ and let $U \subset \Omega^\C_i$ be open. 
Then $\pi^{-1}(\pi(U)) \cap \Omega^\C_j = \varphi_{ji} ( U \cap \Omega_{ij})$, which shows that $\pi^{-1}(\pi(U))$ is open. 

We show that $Z$ is Hausdorff.
If $U, V \subset \Omega^\C_i$ are disjoint, then $\pi^{-1}(\pi(U))$ and $\pi^{-1}(\pi(V))$ are also disjoint.

It therefore only remains to treat the case where $x \in \Omega^\C_1 \backslash \Omega_{12}^\C$ and $y \in \Omega^\C_2 \backslash \Omega^\C_{21}$. 

First assume that $x \in \Omega_1 \subset \Omega^\C_1$ and $y \in \Omega_2 \subset \Omega^\C_2$.
Then $x$ and $y$ can not be in the same $G$-orbit in $X$, since $\Omega_1 \backslash \Omega_{12}^\C$ is a $G$-invariant neighbourhood of $x$ not containing $y$.
Using the existence of slices on $X$, we find open, $G$-invariant disjoint neighbourhoods $U_x$ and $U_y$ of $x$ and $y$ in $X$ with $U_x \subset \Omega_1$ and $U_y \subset \Omega_2$.

Define the open subsets $W_x := G^\C \cdot U_x$ and $W_y := G^\C \cdot U_y$ of $\Omega^\C_1$ and $\Omega^\C_2$. 
Assume that $\pi(W_x)$ and $\pi(W_y)$ are not disjoint. 
Then we find $x_0 \in W_x \cap \Omega^\C_{12}$ and $y_0 \in W_y \cap \Omega^\C_{21}$ such that $x_0$ and $y_0$ are equivalent. 
There exists $h \in G^\C$ such that $h x_0 \in \Omega_1 \cap W_x= U_x$ and since $\varphi_{21} (\Omega_1) \subset \Omega_2$, we conclude $h y_0 \in \Omega_2 \cap W_y = U_y$.
But $h x_0 $ and $h y_0$ are equivalent if and only if $h x_0 = h y_0$ in $X$, implying that $U_x$ and $U_y$ are not disjoint, which is a contradiction.

The general case now follows because the sets $W_{h_x x}$ and $W_{h_y y}$ we constructed are $G^\C$-invariant and every $G^\C$-orbit in $\Omega^\C_i$ intersects $\Omega_i$.

Now $Z$ is second countable because the $\Omega^\C_i$ are second countable.
\end{proof} 

The $\Omega^\C_i$ give $Z$ the structure of a complex manifold.
The $G^\C$-action on $Z$ is holomorphic because the $G^\C$-actions on the $\Omega^\C_i$ are holomorphic. 
The CR embeddings $\Phi_i \colon \Omega_i \rightarrow \Omega^\C_i$ extend to a CR embedding $\Phi \colon X \rightarrow Z$.

We summarise our results in the following Theorem.

\begin{theorem}\label{GlobalExtension}
Let $G$ be a subgroup of its universal complexification and $X$ a CR manifold with proper, transversal CR $G$-action.
Then the map $\Phi \colon X \rightarrow Z$ is a $G$-equivariant CR embedding of $X$ into the complex manifold $Z$.

Furthermore, for every $G$-invariant open subset $U \subset X$, the set $(G^\C \Phi(U), \Phi) $ is the universal equivariant extension of $U$.
In particular, the manifold $(Z, \Phi)$ is the universal equivariant extension of $X$.
\end{theorem}
\begin{proof}
It remains to proof the universality condition. 
We will again assume $U = X$ and take a $G$-equivariant CR map $f \colon X \rightarrow Y$.
For every $x \in X$, the map $f$ extends to a $G^\C$-equivariant holomorphic map $F_i \colon \Omega^\C_i \rightarrow Y$ and we define $F \colon Z \rightarrow Y$ via $F(y) = F_i(y)$ if $y \in \Omega^\C_i$.
We need to show that this is well-defined. 
Consider $\Omega^\C_i$ and $\Omega^\C_j$ with $\Omega_i \cap \Omega_j \neq \emptyset$. 
But then $F_i \colon \Omega^\C_{ij} \rightarrow Y$ and $F_j \circ \varphi_{ji} \colon \Omega^\C_{ij} \rightarrow Y$ are both extensions of the same function $f$ on $\Omega_i \cap \Omega_j$, hence $F_i = F_j \circ \varphi_{ji}$ and $F$ is well-defined.
\end{proof}

We now consider the quotient $X/G$, using that $X/G = Z/G^\C$, and show that it has the structure of a complex space. 
Since the action of $G^\C$ on $Z$ is generally not proper, we need to proof the existence of holomorphic slices, i.e. for every $z \in Z$, there exists a $(G^\C)_z$-invariant complex submanifold $S_C$ of $Z$ with $z \in S_C$ such that the map $G^\C \times^{(G^\C)_z} S_C \rightarrow Z$, $[h,w] \mapsto hw$ is biholomorphic onto an open subset of $Z$.

\begin{definition} 
Let $Z$ be a complex manifold with an action of a complex reductive group $K^\C$, where $K$ is a maximal compact subgroup of $K^\C$. An open subset $W \subset Z$ is called $K$-orbit convex if $W$ is $K$-invariant and for all $z \in W$, $\xi \in \mathfrak{k}$ the set $\{ t \in \R \, | \, \textrm{exp}(it\xi) z \in W \}$ is connected.
\end{definition}

If $Z$, $Y$ are complex manifolds with $K^\C$-action, $W \subset Z$ is $K$-orbit convex and $f \colon W \rightarrow Y$ is a $K$-equivariant holomorphic map, then we may define $F \colon K^\C W \rightarrow Y$ by setting $F(kz):= k f(z)$ for $k \in K^\C$ and $z \in W$.
The identity theorem for holomorphic functions then gives that $F$ is well-defined.

\begin{proposition}\label{sliceLemma}   
Let $G$ be a subgroup of its universal complexification $G^\C$ and $X$ a CR manifold with transversal, proper $G$-action.
Let $x \in X$, $L = G_x$ and $\Omega^\C = G^\C \times^{L^\C} S$ a slice extension around $x$. 
Then there exists a holomorphic slice $S_C$ for the $G^\C$-action on $\Omega^\C$.
\end{proposition}
\begin{proof}
The group $G^\C$ is a Stein manifold \cite[§1 Proposition]{HCom} and $L^\C$ is complex reductive, hence $G^\C / L^\C$ is a Stein manifold \cite{Matsushima}.
We have an $L^\C$-action on $G^\C/L^\C$ via the multiplication from the left.
Since $L^\C$ is complex reductive and $1 \cdot L^\C \in G^\C / L^\C$ is a fixed point, we find an $L^\C$-invariant neighbourhood $U$ of $1 \cdot L^\C$ and a $L^\C$-equivariant embedding of $U$ onto an open subset $U_V$ of an $L^\C$-representation $V$ \cite[Theorem 5.2 and Remark 5.4]{Snow}.
Since $1 \cdot L^\C$ is a fixed point, translation by $1 \cdot L^\C \in U_V$ is $L^\C$-equivariant and we may assume that $1 \cdot L^\C$ gets mapped to $0 \in V$. 

Now $L$ acts by unitary transformations on $V$ for a suitable inner product.
Applying \cite[§3.4 Proposition]{HeinznerGIT} to the inner product, we conclude that $0 \in V$ has a basis of $L$-orbit convex neighbourhoods.\\

The inverse image of an $L$-orbit convex neighbourhood of $1 \cdot L^\C$ via the $L^\C$-equivariant quotient map $G^\C / (L^0)^\C \rightarrow G^\C/L^\C$  is an $L$-orbit convex neighbourhood $\tilde{W}$ of $1 \cdot (L^0)^\C$ in $G^\C /(L^0)^\C$ which is invariant under the $L^\C$-action from the right.

Now if $\tilde{W}$ is as above and $W_0$ is an $L^\C$-invariant neighbourhood of $x \in S$, then the image of $\tilde{W} \times W_0$ is an $L$-orbit convex neighbourhood of $x$ in $G^\C \times^{L^\C} S$. 
We conclude that $x \in G^\C \times^{L^\C} S$ has an $L$-orbit convex neighbourhood basis.

Now take some $L$-orbit convex neighbourhood $W_1$ of $x$ and an open $L$-invariant neighbourhood $W_2$ of $0 \in T_x \Omega^\C$ with an $L$-equivariant biholomorphic map $\varphi \colon W_1 \rightarrow W_2$.
We may extend $\varphi$ to a $L^\C$-equivariant map $\Phi \colon L^\C W_1 \rightarrow L^\C W_2$.

Apply \cite[§3.4 Proposiiton]{HeinznerGIT} on the origin in $T_x \Omega^\C$ and find some $L$-orbit convex neighbourhood $\tilde{W}_2 \subset W_2$. Then extend $\varphi^{-1}$ to a $L^\C$-equivariant map $\tilde{\Phi} \colon L^\C \tilde{W}_2 \rightarrow L^\C W_1$. 
Set $U_1 = L^\C \tilde \Phi (\tilde W)$ and $U_2 = \tilde W$, then $\tilde \Phi$ is inverse to $\Phi \colon U_1 \rightarrow U_2$ because of equivariance.

By choosing an $L^\C$-invariant subspace of $T_x \Omega^\C$ which is perpendicular to $\mathfrak{l}^\C (x)$, we find an $L^\C$-invariant complex submanifold $S_C$ of $\Omega^\C$ through $x$ such that the induced map $\eta \colon G^\C \times^{L^\C} S_C \rightarrow \Omega^\C$ is an immersion in $[1,x]$.

We find some $L$-invariant neighbourhood $\tilde{W}$ of $x$ in $\Omega^\C$ and an $L$-equivariant holomorphic map $\tilde{\eta} \colon \tilde{W} \rightarrow G^\C \times^{L^\C} S_C$ with $\eta \circ \tilde{\eta} =$ Id.

Write $\Omega^\C = G^\C \times^{L^\C} S$, then $\tilde{\eta}|_{S \cap \tilde{W}}$ is an $L^\C$-equivariant map, which extends to a $G^\C$-equivariant map $\bar{\eta}$ on the open set $G^\C \cdot (\tilde{W} \cap S)$.
Because of equivariance, we get $\eta \circ \bar{\eta} = \textrm{Id}$ on $G^\C \cdot (\tilde{W} \cap S)$ and $\eta$ is biholomorphic after possibly shrinking $S_C$.
\end{proof}

\begin{theorem}\label{QuotientComplex}
Let $G$ be a closed subgroup of its universal complexification and $X$ a CR manifold with transversal, proper CR $G$-action.
Then $X/G$ is a complex space, such that the sheaf of holomorphic functions on $X/G$ is given by the sheaf of $G$-invariant CR functions on $X$.
\end{theorem}
\begin{proof}
Let $x \in X$ and let $\Omega^\C$ be a slice extension.
Proposition \ref{sliceLemma} then implies that locally, $\Omega/G = \Omega^\C / G^\C = S_C / L^\C $, where $S_C$ may be realised as an open subset of an $L^\C$-representation $V$. Now $L^\C$ acts as a finite group on $S_C$ and $V$, hence $V/L^\C$ is an affine variety, giving $X/G$ the structure of a complex space.

Because of the universality condition, every $G$-invariant CR map on $\Omega$ extends to a unique $G^\C$-invariant holomorphic function on $\Omega^\C$.
\end{proof}

\begin{remark}
We have actually shown that $X/G$ is a complex orbifold.
\end{remark}

\begin{remark}
If $X$ is a CR manifold with proper, transversal and locally free $G$-action, then the $G$-action on $G^\C \times X$ is free, proper and transversal, hence $G^\C \times^G X$ is a complex manifold. 
Writing $X = G \times^G X$ and applying Lemma \ref{GeneralEmbedding}, one gets that $X \rightarrow G^\C \times^G X$ is an embedding. 
From the construction, it follows that $G^\C \times^G X$ is the universal equivariant extension of $X$.
\end{remark}

\section{Pseudoconvexity}\label{SectionPseudoconvex}

In general, we need to impose additional conditions on a CR manifold to ensure the existence of an embedding into $\C^n$.

In the presence of a $G$-action, we will generalise the notion of pseudoconvexity to CR manifolds of higher codimension and establish a link between strongly pseudoconvex CR manifolds and positive line bundles.

\begin{definition}
Let $X$ be a CR manifold of dimension $(2n,d)$ with an action of a Lie group $G$. We say that the action of $G$ on $X$ is \textbf{transversal of codimension one} if $\textrm{dim} \, \g(x) = d-1$ for all $x \in X$ and 
\begin{gather*}
\g(x) \cap (T^{1,0}_x X \oplus T^{0,1}_x X) = \{ 0 \}.
\end{gather*}
\end{definition}

\begin{example} 
Let $G$ be a Lie group and $M$ a CR manifold of dimension $(2n,d)$ with transversal CR action of $G$. 
Let $X$ be a $G$-invariant hypersurface of $M$, then $X$ is a CR manifold of dimension $(2(n-1), d+1)$ with $G$-action which is transversal of codimension one.

To see this, define $W_x = (T_x^{1,0} M \oplus T_x^{0,1} M) \cap T_xM$, then 
\begin{gather*}
\textrm{dim} W_x + \textrm{dim} T_xX  = \textrm{dim} (T_xX + W_x) + \textrm{dim} ( T_xX \cap W_x).
\end{gather*}
Since $X$ is $G$-invariant, $T_x X$ does contain $\g(x)$ and $T_x X + W_x = T_x M$. 
This implies $\textrm{dim} ( T_xX \cap W_x) = 2n -1$.

Now every $W_x$ has a complex structure, which we denote by $J_x$.
Define $V_x := T_x X \cap W_x$, then
\begin{gather*}
\textrm{dim} (V_x \cap J_x V_x) = \textrm{dim} \, V_x + \textrm{dim} \, J_x V_x - \textrm{dim} ( V_x + J_x V_x),
\end{gather*}
hence $\textrm{dim} (V_x \cap J_x V_x) = 2 (n-1)$.
\end{example}
\smallskip 

Let $X$ be a CR manifold with locally free action of a Lie group $G$ which is transversal of codimension one and assume that $X$ is orientable.

Let $L$ be a line bundle transversal to $\g (x) \oplus (T^{1,0}_x X \oplus T^{0,1}_x X ) \cap T_x X$ in $TX$. 
Since $(T^{1,0} X \oplus T^{0,1} X)\cap TX $ is a complex vector space bundle over $X$, it is orientable.
Since the action is locally free, the bundle $\bigcup_x \g (x)$ is trivial and $L$ is orientable, hence trivial.

We conclude that in this case, we always find transversal vector fields, i.e. a vector field $T$ on $X$ such that $\C T (x) \oplus \C \g (x) \oplus T^{1,0}_x X \oplus T^{0,1}_x X = \C T_x X$.

In fact, the above argument shows that, under the above conditions on the group action, $X$ is orientable if and only if there exist transversal vector fields.\\

From now on, let $X$ be a CR manifold with group action of $G$ which is transversal of codimension one and let $T$ be a transversal vector field on $X$. 
We define a $1$-form on $X$ via $\alpha (T) = 1$ and $\alpha (\g (x) \oplus T_x X \cap (T^{1,0} X \oplus T^{0,1} X )) = 0$, which we will call the \textbf{projection onto $T$}.
For $V,W \in T^{1,0} X$, we define $\omega (V, W) := \frac{1}{2i} d \alpha (V,\overline{W})$, by extending $d \alpha$ to a $\C$-bilinear map on every fibre of $\C T X$.
By definition, $\omega$ is a hermitian form on $T^{1,0} X$.

\begin{definition}
Let $T$ be a transversal vector field on $X$, the form $\alpha$  the projection onto $T$ and $\omega$ be the induced hermitian form.
We say that $T$ is \textbf{strongly pseudoconvex} if $\omega$ is positive definite.
We call $X$ \textbf{strongly CR-pseudoconvex} if there exists a strongly pseudoconvex vector field $T$ on $X$.
\end{definition}

Note that if X is of CR codimension one with trivial group action, this is the usual definition of strong pseudoconvexity.  

For $V,W \in \mathcal{C}^\infty (X , T^{1,0} X)$, we get 
\begin{gather}\label{FormsFormula}
d \alpha (V, \overline{W} ) = V \alpha (\overline{W}) - \overline{W} \alpha (V) - \alpha ([V, \overline{W}]) = -\alpha ( [V, \overline{W}]).
\end{gather}

Let $S$ be another transversal vector field.
We get that $\alpha (S)$ is pointwise non-vanishing, therefore $\beta :=  \alpha(S)^{-1} \cdot \alpha$ is the projection onto $S$. 
Equation (\ref{FormsFormula}) then implies
\begin{gather}\label{ChangeOfVF}
\frac{1}{2i} d \beta (V,\overline{W}) = \frac{1}{2i} \alpha(S)^{-1} \cdot d \alpha (V,\overline{W}).
\end{gather}

\begin{definition}
We say that a vector field $T$ on a CR manifold $X$ is a \textbf{CR vector field} if the flow of $T$ acts by CR automorphisms.
\end{definition}

Note that in general, there do not need to exist transversal CR vector fields. 

Let $X$ be a CR manifold with action of a compact Lie group $K$ which is transversal of codimension one. 
Let $T$ be a transversal vector field and $\alpha$ the projection onto $T$.

The $K$-action leaves the bundle $ \C \k (x) \oplus T^{1,0}_x X \oplus T^{0,1}_x X$ invariant, hence for $k \in K$, we get $\alpha ( d_x k (T(x))) \neq 0$.
We say that $K$ preserves the orientation of $T$ if $\alpha (d_x k( T(x))) >0$ for all $k \in K$ and $x \in X$.
Note that if $K$ is connected, this is always the case.

\begin{lemma}
Let $X$ be a CR manifold with $K$-action which is transversal with codimension one.
Let $T$ be a transversal CR vector field such that $K$ preserves the orientation of $T$.
Then there exists a transversal CR vector field $\tilde T$ which is $K$-invariant.
If $T$ is strongly pseudoconvex, then also $\tilde {T}$ is strongly pseudoconvex.
\end{lemma}
\begin{proof}
Let $\tilde{T}(x) := \int_K d k  (T (k^{-1}x)) dk$ be the vector field averaged over $K$.
If $\alpha$ is the projection onto $T$, then $\alpha (\tilde{T}) = \int_K \alpha (d k T(k^{-1}x) ) dk$.
Because $K$ preserves the orientation of $T$, we conclude $\alpha( \tilde{T}) > 0$. 
In particular, we get that $\tilde T$ is transversal.

The flow of $\tilde{T}$ is given by $\Phi^{\tilde{T}}_t (x)= \int_K k (\Phi^T_t( k^{-1} x)) dk$, where $\Phi^T_t$ is the flow of $T$, which implies that $\tilde{T}$ is CR.
Strong pseudoconvexity of $\tilde T$ follows from Equation (\ref{ChangeOfVF}).
\end{proof}

\begin{theorem}\label{Flow}       
Let $X$ be a compact CR manifold with locally free $K$-action which is transversal of codimension one.
Let $T$ be a strongly pseudoconvex $K$-invariant CR vector field. 
Then there exists a strongly pseudoconvex $K$-invariant CR vector field $\tilde{T}$ such that the flow of $\tilde{T}$ defines an $S^1$-action.
\end{theorem}
\begin{proof}
Let $\alpha$ be the projection onto $T$. We will define a riemannian metric on $X$.
Since the $K$-action is locally free, the bundle $\k (x)$ is trivial and we choose a riemannian metric $h$ for it. 

A vector $V \in T_x X$ can be decomposed as $V = V_T + V_\k + V_\C + V_{\overline{\C}}$, where $V_\C \in T^{1,0}_xX$, $V_{\overline{\C}} \in  T^{0,1}_xX$, $V_T \in \R T(x)$ and $V_\k \in \k(x)$.
Define a riemannian metric via
\begin{gather*}
g(V,U) = \alpha(V_T) \alpha(U_T) + h(V_\k, U_\k) +\frac{1}{2i}  d\alpha( U_\C , \overline{V}_\C) - \frac{1}{2i} d \alpha(U_{\overline{\C}}, \overline{V}_{\overline{\C}} ).
\end{gather*}
Since $\frac{1}{2i}  d\alpha( U_\C , \overline{V}_\C)$ is positive definite on $T^{1,0} X$, the form $\frac{1}{2i}  d\alpha( U_{\overline{\C}} , \overline{V_{\overline{\C}}}) = - \frac{1}{2i}  d\alpha( \overline{V_{\overline{\C}}} , U_{\overline{\C}})$ is negative definite on $T^{0,1}X$.\\

Denote by $\Phi^T_t$ the flow of $T$.
Since $T$ is $K$-invariant and CR, the flow $\Phi^T_t$ commutes with the $K$-action and leaves the bundles $T^{1,0}X$ and $T^{0,1}X$ invariant. 
This implies that $\alpha$ is invariant under $\Phi^T_t$ and that the flow acts by isometries with respect to $g$.\\

The isometry group $\textrm{Iso}(X)$ of $(X,g)$ acts properly on $X$ and since $X$ is compact, we conclude that $\textrm{Iso}(X)$ is compact.
Then the subgroup $U$ of $\textrm{Iso}(X)$ given by the CR isometries which commute with the $K$-action is compact, as well.

Now $T$ defines a subgroup of $U$, which is contained in a torus and the result follows because the vector fields defining $S^1$-actions are dense in this torus.
\end{proof}

Now let $X$ be a compact CR manifold with transversal, locally free CR action of $S^1 \times K$.
Denote $H = S^1 \times K$, then $Y := H^\C \times^H X$ is a complex manifold.
Now consider the CR manifold $(\C \times K^\C) \times X$ with $S^1 \times K$-action given by $((s,k),(z,h,x)) \mapsto (zs^{-1}, hk^{-1}, skx)$.
This action is transversal, therefore $Z := (\C \times K^\C) \times^H X :=  ((\C \times K^\C) \times X) / H$ is a complex space. 
With the same argument, $Z/K^\C = \C \times^{S^1}( X/K)$ is a complex space, as well.
The natural map $Y \rightarrow Z$ is an open holomorphic injective immersion.\\

From Proposition \ref{sliceLemma}, we conclude the existence of holomorphic slices on $H^\C \times^H X$.
The following theorem will be the main motivation for our further study.

\begin{theorem}\label{OrbifoldBundle}     
Let $X$ be a compact CR manifold with transversal, locally free, CR $S^1 \times K$-action such that the vector field $T$ induced by the $S^1$-action is strongly pseudoconvex.
Then the set 
\begin{gather*}
W := \{ [z,x] \in \C \times^{S^1} X/K =  Z/K^\C \, | \, |z|^2 < 1 \}
\end{gather*}
 is strongly pseudoconvex in $Z/K^\C$.
\end{theorem}
\begin{proof}
For the map $h \colon Z/K^\C \rightarrow \R$, $[z,x] \mapsto |z|^2 $, we have $W = \{ p \in Z/K^\C | h(p) < 1 \}$.

Let $U$ be the preimage of $W$ under $Z \rightarrow Z/K^\C$. We pull back $h$ to a function on $Z$, which we will still denote by $h$.
We have $U = \{ z \in Z | h(z) < 1\}$.
The map $x \mapsto [1,1,x]$ is a CR embedding of $X$ and $T$ extends to a $\C^* \times K^\C$-invariant vector field on $Y$, which is the vector field induced by the $S^1$-action on $Y$.

The boundary $\partial{U} = \{ [1, g,x] \in (\C^* \times K^\C) \times^{S^1 \times K} X \} = K^\C X$ is a hypersurface and therefore a CR submanifold of $Y$ with transversal $S^1$-action. 
We therefore have $\C T_x \partial U = T^{1,0}_x \partial U \oplus T^{0,1}_x \partial U \oplus \C T(x)$ and define the projection onto $T$ via $\alpha (T(x)) = 1$ and $\alpha (T^{1,0} \partial U \oplus T^{0,1} \partial U) = 0$.\\

Take $y_0 \in \partial W$ and let $\pi \colon Z \rightarrow Z/K^\C$ be the quotient map. 
We then find an $x_0 \in X$ such that $\pi(x_0) = y_0$.
Let $S$ be a holomorphic slice through $x_0$ for the $H^\C$-action and $\Omega = H^\C \times^L S$, where $L$ is the finite isotropy of $H^\C$ in $x_0$. 
We may assume that $S$ is biholomorphic to a ball and $\C T_{x_0} S = T^{1,0}_{x_0} X \oplus T^{0,1}_{x_0} X$, since actions of compact groups at fixed points can be linearised. 

Let $[w,g,y]$ be the coordinates of $(\C^* \times K^\C) \times^L S$ and take some smooth function $\varphi$ on $\Omega$ such that $h[w,g,y] \cdot e^{\varphi{[w,g,y]}} = |w|^2$.
By construction, $\varphi$ is $H^\C$-invariant and $\Omega \cap \partial{U}$ is given by $\Phi [w,g,y] := |w|^2 - e^{\varphi[w,g,y]} = 0$.\\

Let $\frac{\partial}{\partial w}$ be the complex vector field induced be the $\C^*$-component in $\Omega$, let $(\frac{\partial}{\partial z_j})_{j= 1,...,m}$ be a basis for $T^{1,0} S$ and extend this to a basis $(\frac{\partial}{\partial w}, \frac{\partial}{\partial z_j})_{j = 1,...,n}$ for $T^{1,0} \Omega$.
This is always possible after shrinking $\Omega$.

We want to compute the CR structure of $\partial{U} \cap \Omega$ in terms of the above basis and are therefore considering the equation 
\begin{equation*}
0 = d \Phi ( \sum_j a_j \frac{\partial}{\partial z_j} + b \frac{\partial}{\partial w}) =- \sum_j a_j e^\varphi \frac{\partial}{\partial z_j} \varphi  + b\overline{w}.
\end{equation*}
 This implies that $Z_j = \frac{\partial}{\partial z_j} + \frac{w}{|w|^2} e^\varphi \frac{\partial}{\partial z_j} \varphi \frac{\partial}{\partial w}$ gives a basis for the CR structure.

Rewriting $\frac{\partial}{\partial w}$ in Polar coordinates $(r,\theta)$ gives 
\begin{equation*}
\frac{\partial}{\partial w} = \frac{1}{2} \frac{\overline{w}}{|w|} \frac{\partial}{\partial r} - \frac{1}{2}i \frac{\overline{w}}{|w|^2} \frac{\partial}{\partial \theta}.
\end{equation*}
We get $Z_j = \frac{\partial}{\partial z_j} + \frac{1}{2 |w|}\frac{\partial}{\partial z_j} \varphi \frac{\partial}{\partial r} - \frac{1}{2}i \frac{\partial}{\partial z_j} \varphi \frac{\partial}{\partial \theta}$, using that $\frac{e^\varphi}{|w|^2} = 1$ on the boundary of $U$.
On $\partial U$, the projection $\alpha$ is given by $\alpha( \frac{\partial}{\partial \theta}) = \alpha (T)= 1$ and $\alpha(Z_j) = \alpha ( \overline{Z}_j) = 0$, therefore 
\begin{equation*}
\alpha = d \theta + \sum_j \frac{1}{2} i \frac{\partial}{\partial z_j} \varphi dz_j - \sum_j \frac{1}{2} i \frac{\partial}{\partial \overline{z}_j} \varphi d \overline{z}_j.   
\end{equation*}
Now $\C T_{x_0} S =T^{1,0}_{x_0} X \oplus T^{0,1}_{x_0} X$, hence for $v,w \in T^{1,0}_{x_0} X$, we have the equation 
\begin{equation*}
\frac{1}{2i} d \alpha (v, \overline{w})= -\sum_{ij} \frac{1}{2}  \frac{\partial^2}{ \partial z_j \partial \overline{z}_i} \varphi (v, \overline{w}),
\end{equation*}
which implies that $-\varphi$ is strictly plurisubharmonic in a neighbourhood of $x_0$ in $S$.
After shrinking $S$, we may assume that $- \varphi $ is strictly plurisubharmonic on $S$.

We restrict $\varphi$ to an $L$-invariant function on $S$, write $\Omega / K^\C = \C^* \times^L S$ and 
\begin{gather*}
W \cap ( \Omega/ K^\C) = \{ [w,y] \in \C^* \times^L S \, | \, |w|^2 < e^{\varphi(y)} \}.
\end{gather*}

Define the map $\pi_L \colon \C^* \times S \rightarrow \C^* \times^L S$, then the set $\pi_L^{-1}(W) = \{ (w,y) \, | \, |w|^2 < e^{\tilde{\varphi}(y)} \}$ is a Hartogs domain, which is strongly pseudoconvex because $-\varphi$ is strictly plurisubharmonic and $S$ is biholomorphic to a ball.
Now the boundary of $\pi^{-1}_L(W)$ in $\C^* \times S$ is given by a strongly plurisubharmonic $L$-invariant map $\rho$.
Then $\rho$ defines a map $\overline{\rho}$ on $\C^* \times^L S$, which is plurisubharmonic \cite[Satz 3]{GrauertRemmert} and it only remains to prove that it remains plurisubharmonic under perturbations.
But this follows using \cite[Satz 3]{GrauertRemmert} again and the fact that $\pi_L$ is proper. 
\end{proof}

\begin{remark}
We have actually shown that the orbifold line bundle $Z/K^\C \rightarrow Y/( \C^* \times K^\C)$ is weakly negative (see Section \ref{ProjectiveEmbedding}).
\end{remark}

\section{Equivariant Embeddings}

In this section we prove an embedding theorem for strongly pseudoconvex CR manifolds.

Let $H$ be a subgroup of its universal complexification $H^\C$ and assume that $H^\C$ is complex reductive.
Let $G$ be a normal subgroup of $H$ such that $H = G \rtimes S^1$.
This is a slightly more general setting then the direct product $K \times S^1$ we considered before.
The group $H^\C$ is isomorphic to $G^\C \rtimes \C^*$ and $\C^*$ is a closed subgroup of $H^\C$.

For $h \in H$, we denote by $h_S \in H / G \cong S^1$ the image of $h$ in the quotient and for $h \in H^\C$ we write $h_C \in H^\C /G^\C \cong \C^*$, respectively.
Note that the maps $h \mapsto h_S$ and $h \mapsto h_C$ are group morphisms.

Now let $X$ be a CR manifold with transversal, proper CR action of $H$ and assume that $H_x^0 < G_x$ in every point $x \in X$.
Denote by $T$ the vector field induced by the $S^1$-action.
The condition on the isotropy groups is equivalent to $\C T(x) \oplus \C \g(x) \oplus T^{1,0}_x X \oplus T^{0,1}_x X = \C T_x X$.

Note that the condition $H_x^0 < G_x^0$ is always satisfied if the $H$-action is locally free.\\

We have an $H$-action on the CR manifold $\C \times X$ given by $(h, (z, x)) \mapsto (zh_S^{-1}, h x)$.
This action is proper and CR and $H_x^0 < G_x^0$ implies that it is transversal. 
We conclude that 
\begin{gather*}
Z := (\C \times X) / H \cong \C \times^{S^1} (X/G)
\end{gather*}
is a complex space.

We may embed $X$ into its universal equivariant extension $Y$ as in Theorem \ref{GlobalExtension}.
We define the complex space $( \C \times Y) / H^\C = \C \times^{\C^*} (Y/G^\C)$ in the same manner as above and show that the natural map $(\C \times X)/H \rightarrow (\C \times Y)/H^\C$ is an isomorphism.

It is sufficient to consider a slice neighbourhood $\Omega = H \times^{H_x} S$ of $X$ and show that the map $\varphi \colon (\C \times \Omega) /H \rightarrow (\C \times \Omega^\C ) / H^\C$ is biholomorphic.
But the inverse of $\varphi$ is given by $\varphi^{-1} \colon (\C \times (H^\C \times^{H_x^\C} S)) /H^\C \rightarrow (\C \times (H \times^{H_x} S))/H$ with $[z,[h,s]] \mapsto [z h_C, [1,s]]$.
  
\begin{definition}
We say that $X$ is \textbf{strongly pseudoconvex} if the set 
\begin{gather*}
Z_\Delta := \{ [w,z] \in \C \times^{S^1} X/G \, | \, |w| <1 \}
\end{gather*}
is strongly pseudoconvex in $Z$.
\end{definition}

The following construction, including Corollary \ref{VanishingCohomology}, is based on ideas of Grauert, see \cite[§3.2]{GrauertEx}.
We will reformulate it here, in our context.

\begin{definition}
Let $M$ be a set with $S^1$-action, $U \subset M$ a subset and $V$ a vector space. We say that a map $f \colon U \rightarrow V$ is of \textbf{order} $d \in \mathbb{Z}$ if for all $s \in S^1$ and $m \in U$ such that $sm \in U$, we have $f(s m) = s^d f(m)$.
\end{definition}

Consider the complex space $X/H = Y/H^\C$.
The structure sheaf $\O^H_Y$ of $Y$ is given by the sheaf of $H$-invariant holomorphic functions on $Y$.
Given a coherent analytic sheaf $\G^Y$ over $Y/H^\C$, we define a sheaf $\G^Z$ over $Z$ as follows.
The holomorphic map $p \colon Z \rightarrow Y/H^\C$ gives the structure sheaf $\O_Z$ of $Z$ the structure of an $\O^H_Y$-module.
Define
\begin{gather*}
\G^Z (U) = \G^Y ( p(U)) \otimes_{\O^H_Y(p(U))} \O_Z ( U ).
\end{gather*}
Now $\G^Z$ is coherent and analytic on $Z$ because $\G^Y$ is coherent and analytic on $Y/H^\C$.\\

The map $Y/H^\C \rightarrow Z$, $y \mapsto [0,y]$ is an embedding of $Y/H^\C$ as the analytic set of $\C^*$-fixed points. 
Define the coherent sheaf $\O_Z^Y$ of holomorphic functions on $Z$ which vanish on $Y/H^\C$. 
We have the (non-analytic) subsheaf $\O_{Z,d}^Y$ of functions vanishing on $Y/H^\C$ which are of order $d$.
The sheaves $\O_Z^Y$ and $\O_{Z,d}^Y$ then define analytic sheaves $\O_Y^Y (U)$ and $\O_{Y,d}^Y (U)$ on $Y/H^\C$ via pulling back with $p$.

Let $\G^Y$ be a coherent analytic sheaf over $Y/H^\C$.
The natural map $\O^Y_{Y,d}(U) \rightarrow \O_Z (p^{-1} (U))$ induces the map 
\begin{gather*}
A_{\leq d_0} \colon \bigoplus_{0 \leq d \leq d_0} H^q(Y/H^\C, \G^Y \otimes \O^Y_{Y,d}) \rightarrow H^q(Z, \G^Z),
\end{gather*}
where $H^q$ denotes the $q$-th \v{C}ech cohomology.\\

We now define a map
\begin{gather*}
B_d \colon H^q(Z, \G_Z) \rightarrow H^q(Y/H^\C, \G^Y \otimes \O^Y_{Y,d})
\end{gather*}
as follows.

Let $U \subset Z$ be open with $U \cap Y/H^\C \neq \emptyset$ and $g \colon U \rightarrow \C$ a holomorphic function.

Let $\pi \colon \C \times Y \rightarrow (\C \times Y )/H^\C=Z$ be the quotient map, then $\overline g := g \circ \pi$ is a $H^\C$-invariant holomorphic function on $\pi^{-1}(U)$.
Consider the set $U_0 = \{ (0,w) \in \pi^{-1}(U) \}$, take $(0,w_0) \in U_0$ and develop $\overline g$ into a power series $\overline g = \sum_d z^d f_{d,w_0} (w)$ in some connected neighbourhood of $(0,w_0)$ of the form $B_{w_0} \times \Omega_{w_0}$.

If $(0, w_1) \in U_0$ is another point such that $B_{w_1} \times \Omega_{w_1} \cap B_{w_0} \times \Omega_{w_0} \neq \emptyset$, we have $f_{d,w_0} = f_{d, w_1}$ on the intersection.
We may therefore assume that $f_d$ is defined on $\Omega := \bigcup_{(0,w) \in U_0} B_w \times \Omega_w$ and set $\overline g_d (z,w) := z^d f_d (w)$ on $\Omega$. 

Now the $\overline g_d$ are actually $H^\C$-invariant, which we will see as follows.
Set $\tilde g_d (hx) := \overline g_d (x)$ for $h \in H^\C$, $x \in \Omega$, we need to show that this is well-defined.
For this, fix some $h \in H^\C$ and observe that every connected component of $h \Omega$ intersects $U_0$.
For $(0,w_0) \in U_0$, we consider the power series expansions of $\overline g$ and $\overline g \circ h^{-1}$, which have to coincide because $\overline g$ is $H^\C$-invariant. 
For $c \in \C^*$, we compute
\begin{gather*}
\overline g_d(h^{-1} (cz,w)) = \overline g_d (cz h_C, h^{-1}w) = c^d \overline g_d (zh_C, h^{-1}w).
\end{gather*}
Comparing coefficients in the power series expansions then gives $\overline g_d = \overline g_d \circ h^{-1}$ on $\Omega \cap h\Omega$ and $\tilde g_d$ is well-defined.

We may therefore assume that the sets $B_w \times \Omega_w$ are $H^\C$-invariant. Since the $\overline g_d$ are of order $d$, they may be extended to functions on $\bigcup_{(0,w) \in U_0}\C \times \Omega_w$.
Now the $\overline g_d$ define holomorphic functions $g_d$ on $p^{-1} (p( U \cap Y/H^\C))$ of degree $d$.\\

For every open subset $U \subset Z$ with $U \cap Y/H^\C \neq \emptyset$, we have a map 
\begin{gather*}
B_{d,U} \colon \G^Z (U) \rightarrow \G^Y( p(U \cap Y/H^\C)) \otimes \O^Y_{Y,d}(p(U \cap Y/H^\C))
\end{gather*}
by restriction in the first component and the construction from above in the second component.

Now we want to define $B_d$ on the cohomology groups.

If $V$ is an open subset of $U$ with $V \cap Y/H^\C \neq \emptyset$, then $V \cap Y/H^\C$ is a subset of $U \cap Y/H^\C$ and from the above construction,  we see 
 \begin{gather}\label{CommuteRes}
 B_{d,V}( \textrm{res}^U_V (g) )=\textrm{res}^{p(U \cap Y/H^\C)}_ {p(V \cap Y/H^\C)} B_{d,U} (g).
 \end{gather}

Given a covering $\mathcal{V}$ of $Z$, we define a covering $\mathcal{V}^Y$ of $Y/H^\C$ via $V_i^Y :=  p (V_i \cap Y/H^\C)$, \\if $V_i \cap Y/H^\C \neq \emptyset$. 
For a cochain $g \in C^q ( \mathcal{V}, \G^Z)$ and a $q$-simplex $\sigma^Y$ for $\mathcal{V}^Y$, with corresponding simplex $\sigma$ for $\mathcal{V}$, we define $B_d(g) ( \sigma^Y) := B_{d, |\sigma|} ( g(\sigma) )$.

Now if $G \in C^{q-1}( \mathcal{V}, \G^Z)$ is a cochain, $\sigma$ is a $q$-simplex for $\mathcal V$ and $\partial_j \sigma$ is the corresponding $(q-1)$-simplex by omitting the $j$-th set, then we have 
\begin{equation*}
\begin{split}
B_{d, |\sigma|} ( \delta G (\sigma)) &= B_{d, |\sigma|} \big( \sum_j (-1)^j \textrm{res}^{|\partial_j \sigma|}_{|\sigma|} (G(\partial_j \sigma)) \big) \\
&=\sum_j (-1)^j \textrm{res}^{p(|\partial_j \sigma| \cap Y/H^\C)}_{p(|\sigma| \cap Y/H^\C)} (B_{d, |\partial_j \sigma|}(G(\partial_j \sigma)) ).
\end{split}
\end{equation*}
From this computation and equation (\ref{CommuteRes}), we conclude that the maps
\begin{gather*}
B_d \colon H^q( Z, \G^Z) \rightarrow H^q ( Y/H^\C, \G^Y \otimes \O^Y_{Y,d}),
\end{gather*}
and
\begin{gather*}
B_{\leq d_0} \colon H^q( Z, \G^Z) \rightarrow \bigoplus_{d \leq d_0} H^q ( Y/H^\C, \G^Y \otimes \O^Y_{Y,d}).
\end{gather*}
are well-defined, since they map cocycles to cocycles and coboundaries to coboundaries.

Since $B_{\leq d_0} \circ A_{\leq d_0} = $Id, we conclude that $A_{\leq d_0}$ is injective for every $d_0$.\\

The following Proposition was proved by Grauert \cite[Proposition 4]{GrauertLevi} for the structure sheaf on complex manifolds and the proof also holds for this case using \cite[Theorem 5.4]{Fornaess}, also see \cite[Hilfssatz 1]{GrauertEx}. 

\begin{proposition}\label{CohomologyFinite}
Let $Z$ be a complex space and $\Omega \subset Z$ a strongly pseudoconvex, relatively compact set.
Let $\G$ be a coherent analytic sheaf on $Z$.
Then the complex vector spaces $H^q (\Omega, \G)$ are finite-dimensional for $q > 0$. 
\end{proposition}

From Proposition \ref{CohomologyFinite} and the injectivity of $A_{\leq d_0,}$, we conclude the following result.

\begin{corollary}\label{VanishingCohomology}  
Let $X/G$ be compact, $X$ be strongly pseudoconvex and $\G^Y$ a coherent analytic sheaf on $Y/H^\C$.
Then for every $q >0$, there exists a $d_0$ such that 
\begin{equation*}
H^q(Y/H^\C, \G^Y \otimes \O_{Y,d}^Y) = 0
\end{equation*}
for all $d \geq d_0$.
\end{corollary}

Let $V$ be a $H^\C$-representation. We may define a new $H^\C$-representation structure on $V$ via $(h,v) \mapsto h_C^d h v$, which we denote by $V_d$.\\

Note that for an open subset $U$ of $Y$, the $G^\C$-invariant holomorphic functions of order $d>0$ on $U$ are one-to-one with the holomorphic functions of order $d$ on $\C \times^{\C^*} U /G^\C$.

Now let $\pi \colon Y \rightarrow Y/H^\C$ be the projection and $V$ a $H^\C$-representation. We define the sheaf $\F^Y$ on $Y/H^\C$ as 
\begin{gather}\label{SheafDef}
\F^Y (U):= \{ f \colon \pi^{-1}(U) \rightarrow V \, | \, f \textrm{ holomorphic, $H^\C$-equivariant } \}.
\end{gather}
For $z \in Y/H^\C$ we set 
\begin{gather}\label{SheafOne} 
\F^Y_{z,z} (U) := \{ f \in \F^Y (U) \, | \, f(z) = 0 \textrm{ and } df (z) = 0 \}
\end{gather}
and for $z,w \in Y/H^\C$, consider
\begin{gather}\label{SheafTwo}
\F^Y_{z,w} (U) := \{ f \in \F^Y (U) \, | \, f(z) = f(w ) = 0 \}.
\end{gather}
The existence of holomorphic slices (Proposition \ref{sliceLemma}) also implies that every $H^\C$-orbit is analytic.
This shows that for a fixed $y \in Y$, the sheaf $\O^V_{H^\C y}$ of holomorphic maps to $V$ which vanish on $H^\C y$ is a coherent analytic $H^\C$-sheaf (see \cite{Roberts}).

Given a point $x \in Y$, we find a slice $S_C$ which is a Stein manifold.
Since $H^\C$ is complex reductive, we have that $H^\C \times^{H^\C_y} S_C$ is Stein.
Hence $Y/H^\C$ has a neighbourhood basis such that the preimages under the quotient map are Stein. 
We may use \cite[Theorem 3.1]{Roberts} to see that both sheaves above are coherent analytic sheaves on $Y/H^\C$.\\

Let $x \in X$ and $L= H_x$ with $L^0 \subset G$.
Then $L^\C$ acts from the right on $H^\C / G^\C$ as a finite group $N$.
We call the order $|N|$ of the finite group $N$ the \textbf{order of $\C^*$ in $x$}.
The map $h_C \mapsto h_C^{|N|}$ is an $N$-invariant map on $H^\C / G^\C$.

\begin{proposition}\label{LocalImmersionDependent}
Let $y \in X$, $L = H_y$ and $H^\C \times^{L^\C} S_C$ be a holomorphic slice in $Y$. Let $l$ be  the order of $\C^*$ in $x$. Then there exists a $H^\C$-representation $V $ and a $H^\C$-equivariant holomorphic injective immersion 
$\varphi \colon H^\C \times^{L^\C}S_C \rightarrow V$ such that, for every $m \in \mathbb{N}$, the map 
\begin{gather*}
\Phi \colon H^\C \times^{L^\C} S_C \rightarrow V_{ml} \times V_{(m+1)l} \\
[h,z] \mapsto (h_C^{ml} \varphi[h,z], h_C^{(m+1)l} \varphi[h,z])
\end{gather*}
is a well-defined, injective equivariant immersion on some open, $H^\C$-invariant neighbourhood $U_m$ of $y$ in $Y$.

For all $y_0 \in U_m$, the restricted map $\Phi \colon G^\C y_0 \rightarrow V_{ml} \times V_{(m+1)l}$ is an embedding.
\end{proposition}
\begin{proof}

We May assume that $S_C \subset T_y Y$ and $H^\C \times^{L^\C} S_C \subset H^\C \times^{L^\C} T_y Y$. 
The quotient $H^\C \times^{L^\C} T_y Y$ is an affine variety with an action of the complex reductive group $H^\C$, therefore we find an equivariant holomorphic embedding into some representation $V$.

Let $\varphi \colon H^\C \times^{L^\C} S_C \rightarrow V$ be the restriction of that embedding.
The $L^\C$-orbits in $S_C$ are all finite. We conclude that all $H^\C$-orbits in $H^\C \times^{L^\C} S_C$ are also closed in $H^\C \times^{L^\C} T_y Y$, hence $\varphi$ is an embedding on every $H^\C$-orbit.\\

Let $\Phi$ be defined as in the statement of the Proposition. 
Using that $\varphi$ is an equivariant immersion and therefore non-vanishing, a direct computation shows that $\Phi$ is injective.

Now consider $ \tilde{z} :=[1,z]$, take $\xi \in\C = T_1 \C^*$, $\nu \in \G^\C$ and $v \in T_{\tilde{z}} S_C$.\\
Define $:w =\xi (\tilde{z}) + \nu ( \tilde{z}) + v$, then 
\begin{equation*}
\begin{split}
&d_{\tilde{z}} \Phi (w) = (l m \cdot  \xi \cdot \varphi (\tilde{z}) + d_{\tilde{z}} \varphi (w ), \, l(m+1) \cdot \xi \cdot \varphi (\tilde{z}) + d_{\tilde{z}} \varphi (w ))
\end{split}
\end{equation*}
and because $\varphi(\tilde{z}) \neq 0$ and $\varphi$ is an immersion, one sees that $\Phi$ is an immersion in $\tilde{z}$.
Because $\Phi$ is equivariant, it is an immersion in some invariant neighbourhood $U_m$ of $H^\C y$.

For fixed $y_0 \in U_m$, we get $\Phi|_{G^\C y_0 } = (c_1 \varphi, c_2\varphi)$ for some constants $c_1$ and $c_2$, which is an embedding.
\end{proof}

For fixed $y \in X$ with order $l$ of $\C^*$ in $y$, define $f_m \colon H^\C \times^{L^\C} S_C \rightarrow \C$, $[h,x] \mapsto h_C^{ml}$.
Let $\varphi$ be the map as in Proposition \ref{LocalImmersionDependent}, then $\varphi \otimes f_m$ defines an element in $\F^Y \otimes \O^Y_{Y,ml}$ (see (\ref{SheafDef})). 
Applying Corollary \ref{VanishingCohomology} to the sheaf (\ref{SheafOne}) and choosing $m$ large enough, we find a $H^\C$-representation $V_y$, and a  $H^\C$-equivariant holomorphic map $\Phi_y \colon Y \rightarrow V_y$ such that $\Phi_y = (f_m \varphi, f_{m+1} \varphi)$ on $H^\C y$ up to order $2$.
We therefore find a neighbourhood $U_y$ of $y$, such that the map $\Phi_y \colon U_y \rightarrow V_y$ is an immersion on $U_y$, injective on $H^\C y$ and $\Phi_y \colon G^\C y \rightarrow V_y$ is an embedding.\\

Using the same argument for the trivial $H^\C$-representation $\C$ and $\varphi$ being a constant map, we may construct a $G^\C$-invariant holomorphic map $F \colon Y \rightarrow \C^m$ which does not vanish on $Y$ and every component $F_i$ of $F$ is of order $d_i >0$.
Because $X/G$ is compact, we obtain inf$_{x \in X} ||F(x)||^2 = c > 0$.

\begin{lemma}\label{LocalProper}
Let $\Phi_y$, $F$ be as above.
Then there exist open, $H^\C$-invariant neighbourhoods $U_y$ of $y$ in $Y$ and $W_y$ of $(F \times \Phi_y ) (y)$ in $\C^m \backslash \{0\} \times V_y$ such that $(F \times \Phi_y) \colon U_y \rightarrow W_y$ is an embedding. 
Furthermore, the set $W_y$ is saturated with respect to the quotient map $\C^m \backslash \{0\} \times V_y \rightarrow (\C^m \backslash \{0\} \times V_y) \git H^\C$ and two distinct $H^\C$-orbits in $(F \times \Phi_y) (U_y)$ may be separated by disjoint, $H^\C$-invariant neighbourhoods.
\end{lemma}
\begin{proof}
We will omit the subscript and write $\Phi = \Phi_y$.
We begin by showing that the map $(F \times \Phi) \colon H^\C y \rightarrow \C^m \backslash \{0\} \times V$ is an embedding.

Let $w_n \in \C^*$, $h_n \in G^\C$ such that $(F \times \Phi)(w_n h_n y)$ converges in $\C^m \backslash \{0\} \times V$. 
We get $F(w_n h_n y) = F(w_n y) = \bigoplus w_n^{d_i} F_i (y)$, which converges against a non-zero value.
There exists an $i$ such that $F_i (y) \neq 0$, hence $w_n$ is bounded from below and above, therefore we may assume that it converges in $\C^*$.

But then $w_n^{-1} \Phi(w_n h_n y) = \Phi ( h_n y)$ converges and since $\Phi$ is an embedding on $G^\C y$, we have proved the claim. \\

We get that the $\C^*$-action on $\C^m \backslash \{0\}$ induced by the orders of the $F_i$ is proper and locally free, therefore we find a complex slice $S$ around $F(y)$ which is biholomorphic to some complex ball.
Let $\Gamma = \C^*_{F(y)}$, then $\Omega := \C^* \times^\Gamma S$ is an open, $\C^*$-invariant Stein neighbourhood of $F(y)$ in $\C^m \backslash \{0\}$ and the $H^\C$-orbit through $(F \times \Phi)(y)$ is closed in $\Omega \times V$.

Using \cite[Einbettungssatz 1]{HeinznerEmbedding}, we find a closed $H^\C$-equivariant embedding $\mu$ of $\Omega \times V$ into some $H^\C$-representation $V_0$.
Now $\mu \circ (F \times \Phi)$, defined on $(F \times \Phi)^{-1}(\Omega \times V)$, is an immersion in $y$ and an embedding on the $H^\C$-orbit through $y$.
%%note to myself: this is necessary because there is no obvious slice theorem on $\C^m \backslash \{0\} \times V$ and we may not apply the result of Snow immediately. 

From the proof of Proposition 5.1 in \cite{Snow}, one finds a slice $S_0$ through $\mu \circ (F \times \Phi) (y)$ in $V_0$ and a slice $S_Y$ through $y$ in $Y$ such that the maps $\mu \circ (F , \Phi) \colon S_Y \rightarrow S_0$ and also $\mu \circ (F \times \Phi) \colon H^\C \times^{L^\C} S_Y \rightarrow H^\C \times^{L^\C} S_0$ are embeddings. Additionally, $H^\C \cdot S_0$ is saturated with respect to the quotient map.

Define $W_0 = H^\C \cdot S_0$, $W := \mu^{-1}(W_0)$ and $U := H^\C \cdot S_Y$.
Then $W$ is saturated with respect to the quotient map and $(F \times \Phi) \colon U \rightarrow W$ is an embedding. 

Now every $H^\C$-orbit in $U$ is closed, hence every $H^\C$-orbit in $\mu \circ (F \times \Phi) (U)$ is closed in $W_0$. We may separate two distinct, closed orbits in $W_0$ by $H^\C$-invariant open neighbourhoods and the result follows.
\end{proof}
%%note to myself: it is not possible to state the theorem without the \C^m / {0}-part because the closed map into a representation does not extend to a globally defined map!

Let us shortly summarize the assumptions we made to this point. 
We have that $H$ is a subgroup of its universal complexification $H^\C$, $G$ is a closed, normal subgroup of $H$ such that $H = G \rtimes S^1$.
Furthermore, we assume $H^\C$ to be complex reductive. 
Now let $X$ be a CR manifold with proper, transversal CR action of $H$ such that $H_x^0 \subset G_x^0$ for every $x \in X$ and $Y$ the universal equivariant extension of $X$. 
Our main result is the following.

\begin{theorem}\label{PseudoconvexEmbedding}
Let $X$ be as above and assume that $X/G$ is compact.\\
Then there exists a $H^\C$-representation $V$ and a $H^\C$-equivariant holomorphic embedding $\Phi \colon Y \rightarrow \C^m \backslash \{0\} \times V$, such that $\Phi|_X \colon X \rightarrow \C^m \times V$ is a CR embedding. 
Here, $\C^m$ is the trivial $G^\C$-representation and decomposes into irreducible $\C^*$-representations with positive weights.
\end{theorem}
\begin{proof}
Let $F$, $\Phi_y$, $U_y \subset \Omega_y^\C$ and $W_y$ be as in Lemma \ref{LocalProper}.

Using that $X/G$ is compact, we may cover $Y$ be finitely many open subsets $U_y$ with $y \in I$ and consider
\begin{gather*}
\tilde{ \Phi} := (F, \bigoplus_{y \in I} \Phi_y) \colon Y \rightarrow \C^m \times \bigoplus_{y \in I} V_y,
\end{gather*}
which is an immersion on $Y$ and injective on every $H^\C$-orbit. \\

We write $\tilde{V} := \C^m \times \bigoplus_y V_y$, $\tilde{V}_0 := (\C^m \backslash \{0\}) \times  \bigoplus_y V_y$ and show that the map $\tilde{\Phi} \colon Y \rightarrow \tilde{V}_0$ is proper.

Let $a_n$ be a sequence in $Y$ such that $\tilde{\Phi} (a_n) = b_n$ converges against $b \in \tilde{V}_0$.
Because $Y/ H^\C = X/H$ is compact, we may assume that $\pi (a_n) \rightarrow \pi(a)$ converges, where we write $\pi \colon Y \rightarrow Y/H^\C$ for the quotient map.

Around $a$, we find a neighbourhood $U_y$ and an open neighbourhood $\tilde{W}_y$ of $\tilde{\Phi}(a)$, such that $\tilde{\Phi} \colon U_y \rightarrow \tilde{W}_y$ is an embedding and $\tilde{W}_y$ defines an open neighbourhood of $b$ in $\tilde{V}_0 \git H^\C$.
This implies the convergence of a subsequence of $a_n$.
Since $\textrm{inf}_{x \in X} || F(x)||^2 = c > 0$, we conclude that $\tilde \Phi \colon X \rightarrow \tilde V$ is proper.\\

Since $\tilde \Phi$ is injective on every $H^\C$-orbit it only remains to show that we may separate points on different orbits.

Lemma \ref{LocalProper} states that for every $w,z \in U_y$ in different orbits, we may separate $\tilde{\Phi}(w)$ and $\tilde{\Phi}(z)$ by $H^\C$-invariant open sets. 
For two distinct points $z,w \in Y$ in different $H^\C$-orbits, we may separate $z$ and $w$ by $H^\C$-invariant sets of the form $\Omega_z = H^\C \times^{H^\C_z} S_C^z$ and $\Omega_w = H^\C \times^{H^\C_w} S_C^w$.
For every $d$ that is a multiple of the orders of $\C^*$ in $z$ and $w$, the map $f_d[h,z] \mapsto h_C^d$ is well-defined on $\Omega_z \cup \Omega_w$. 
Let $\varphi \colon \Omega_z \cup \Omega_w \rightarrow \C^2$ be equal to $(1,0)$ on $\Omega_z$ and equal to $(0,1)$ on $\Omega_w$.
Note that $\varphi$ is equivariant with respect to the trivial $H^\C$-representation.

Applying Corollary \ref{VanishingCohomology} on the sheaf (\ref{SheafTwo}) for the map $f_d \cdot \varphi$ and sufficiently large $d$, we find a $G^\C$-invariant map $F_{z,w} \colon Y \rightarrow \C^2$ of order $d > 0$ such that $F(z)$ and $F(w)$ may be separated by open, $H^\C$-invariant sets.
 
Because $(Y/H^\C \times Y/H^\C ) / \bigcup_{y \in I} (U_y/H^\C \times U_y/H^\C)$ is compact, we conclude that we may use this argument finitely many times and the result follows.
\end{proof}

\section{Line Bundles}\label{ProjectiveEmbedding}

We apply the methods used in the previous sections for a proof of an equivariant embedding theorem similar to the Kodaira embedding theorem.\\ 

Let $X$ be a compact CR manifold and $K$ a compact Lie group with transversal CR action on $X$.

We will use the definitions from \cite{HsiaoMarinescu} for rigid, positive CR line bundles, generalised to arbitrary compact groups.

Let $X$ be a CR manifold with transversal $K$-action. We say that a CR line bundle $L \rightarrow X$ is \textbf{$K$-invariant} if there exists a cover $U_i$ of $X$ over which $L$ is trivial such that each $U_i$ is $K$-invariant and the transition functions $g_{ij}$ are $K$-invariant and CR.
The transition functions can be viewed as holomorphic functions on $X/K$ and give rise to a holomorphic line bundle $L_K \rightarrow X/K$.
On the other hand, every holomorphic line bundle $L_K \rightarrow X/K$ also induces a CR line bundle $L \rightarrow X$.
We say that $L$ is positive if $L_K$ is positive.

Note that if $\Omega$ is a slice-neighbourhood of $x$ in $X$ and $f \colon \Omega \rightarrow \R$ is a $K$-invariant smooth map, then $f$ extends to a $K^\C$-invariant map $F$ on $\Omega^\C$. 
Assume now that the hermitian form $(V,W) \mapsto i \partial \overline{\partial} F(V,\overline{W})$ is positive for $V,W \in T^{1,0}_x X$.  	
We find a slice $S_C$ through $x$ in $\Omega^\C$ such that $\C T_x S_C = T^{1,0}_x X \oplus T^{0,1}_x X$, hence $F|_{S_C}$ is strictly plurisubharmonic after possibly shrinking $S_C$.
Since $(K_x)^\C$ acts as a finite group on $S_C$, we conclude that $F|_{S_C}$ defines a strictly plurisubharmonic function on $S_C / (K_x)^\C$, using the same argument as at the end of Theorem \ref{OrbifoldBundle}.
From \cite[Proposition 6.1]{MarinescuMorse}, we then see that if $L$ is a positive, invariant CR line bundle in the sense of \cite{HsiaoMarinescu}, then it is also positive using the definition given above.\\

Note that every negative line bundle with metric $h$ is weakly negative.
This means that the set $\{ z \in L \, | \, h(z) < 1 \}$ is strongly pseudoconvex, see \cite[§3 Satz 1]{GrauertEx}. 

Now let $L \rightarrow X$ be a weakly negative line bundle and $X_S$ the $S^1$-bundle of $L$, which is a strongly pseudoconvex CR manifold. 
Then $\C \times^{S^1} X_S \cong L$ and $\C \times^{S^1} X_S/K \cong L_K$.

If $V$ is a $K^\C$-representation and $\varphi \colon L \rightarrow V$ a CR map of order $d$, then we may define a section $X \rightarrow L^{-d} \otimes V$ by setting $s_\varphi (x) := 1 \otimes \varphi(1,x)$ in a trivialisation of $L$.
A direct computation shows that this does indeed define a global section.

Now let $x \in X$, $\varphi \colon U \rightarrow V$ an equivariant inkective immerion in some neighbourhood $U$ into a $K^\C$-representation $V$ (see proof of theorem \ref{LocalImmersionDependent}).
Denote by $p \colon L \rightarrow X$ the bundle map.
After shrinking $U$, we may assume that $L$ is trivial over $U$ and define the map $f_d \colon p^{-1}(U) \rightarrow \C$, $(w,x) \mapsto w^d$.
Applying Corollary \ref{VanishingCohomology} to the sheaf (\ref{SheafOne}) for the point $(1,[x]) \in L_K$, the map $f_d \cdot (\varphi \circ p)$ and large $d$ gives an equivariant map $\Phi \colon L \rightarrow V_d$ with $\Phi = f_d \cdot ( \varphi \circ p)$ of order $2$ on $(S^1 \times K) \cdot (1,x) \in L$.

The map $\Phi$ induces a section $\varphi_0 \colon X \rightarrow L^{-d} \otimes V$.
Indentifying $L^{-d} \otimes V$ with $V$ on $U$ using $c \otimes v \mapsto cv$, we get $\varphi_0 (x) = \Phi(1,x) = f(1,x) \cdot (\varphi \circ p) (1,x) = \varphi(x)$.
Let $v \in T_x U$, then $d_x \varphi_0 (v) = d_{(1,x)} \Phi (0,v) = d_{(1,x)} (f \cdot  (\varphi \circ p)) (0,v) = d_x \varphi (v)$.

Denote by $\L^k (C)$ the CR sections $X \rightarrow L^{-d}$ and by $\L^k_x(X)$ and $\L^k_{x,y} (X)$ the sections vanishing in $x$ of order $2$ and the sections vanishing in $x$ and $y$, respectively.
For $x \in X$, we may use the argument above and choose a basis for $V$ to obtain sections $s_1,...,s_m \in \L^d(X)$ such that the $d_x s_i$ form a generating system for $(T_x X)^*$.
Using the same argument as above for $\varphi$ being a constant map, we also find a section $s \in \L^d(X)$ with $s(x) \neq 0$ and $d s (x) = 0$.
Using an analogous argument for Corollary \ref{VanishingCohomology} and sheaf (\ref{SheafTwo}), we also conclude that given two points $x,y \in X$ in different $K$-orbits, we find a section $s \in \L^{d_{x,y}} (X)$ such that $s(x) = 0$ and $s(y) \neq 0$. 

We have therefore shown that for every $x \in X$, there exists a $k$ such that the sequence 
\begin{gather}\label{SequenceOne}
0 \rightarrow \L^k_x(X) \rightarrow \L^k(X) \rightarrow (\L^k / \L^k_x)(X) \rightarrow 0
\end{gather}
is exact. Also, given $x,y \in X$, there exists a $k$ such that the sequence 

\begin{gather}\label{SequenceTwo}
0 \rightarrow \L^k_{x,y}(X) \rightarrow \L^k(X) \rightarrow (\L^k / \L^k_{x,y})(X) \rightarrow 0
\end{gather}
is exact.

Now if Sequence (\ref{SequenceOne}) is exact in $x$ for $k$, it is also exact in a neighbourhood of $x$. 
The same applies for Sequence (\ref{SequenceTwo}).
The natural map $\Gamma(X, (L^{-d})^k) \rightarrow \Gamma(X, L^{-dk})$ is surjective, hence if sequences (\ref{SequenceOne}) and (\ref{SequenceTwo}) are exact for $k$, they are also exact for $dk$ for every $d >0$. 

Since $X$ is compact, we conclude that there exists a $k$ such that sequences (\ref{SequenceOne}) and (\ref{SequenceTwo}) are exact for every point in $X$ and pairs of points in $X$, respectively.

\begin{corollary}\label{LineBundleEmbedding}
Let $X$ be a compact CR manifold with a transversal CR action of a compact Lie group $K$. \\
Assume that there exists a weakly negative line bundle $L_K \rightarrow X/K$.
Then there exists a natural number $k$ and finitely many CR sections $s_i \in \L^k(X)$ such that, for
 $W = span( s_i)$, we have that
\begin{gather*}
X \rightarrow \mathbb{P} ( W^* ) \\
y \mapsto [ s \mapsto s(y) ]
\end{gather*}
is a CR embedding.
\end{corollary}
\begin{proof}
For a subspace $W \subset \L^k(X)$, we denote by $\Phi_W$ the map defined in the corollary for this subspace.
We may assume that $\Phi_W$ is always well-defined.

Choose $k$ so that Sequences (\ref{SequenceOne}) and (\ref{SequenceTwo}) are exact for every $x \in X$ and pairs $x,y \in X$.

Take $x \in X$, then we find sections $g_1,...,g_m \in \L^k$ such that $y \mapsto (g_1(y),...,g_m(y))$ is an immersion in $x$ in a trivialisation around $x$.
We also get a section $g_0$ with $g_0(x) \neq 0$ and $d_x g_0 = 0$.

Set $W_x = \textrm{span} (g_0,...,g_m)$. We may assume that the $g_i$ form a basis for $W_x$, then $\Phi_{W_x}(y) = [g_0(y),...,g_m(y)]$ in the corresponding dual basis.
This shows that $\Phi_{W_x}$ is an immersion in $x$.

We may repeat this process finitely many times to find a subspace $W$ of $\L^k(X)$ such that $\Phi_W$ is an immersion on $X$.
We may therefore cover $X$ by finitely many open sets $U_i$ such that $ \Phi_W \colon U_i \rightarrow \mathbb{P}(W^*)$ is injective.

Now $(X \times X) / \bigcup_i (U_i \times U_i)$ is compact, and for two points $(x,y) \in (X \times X) / \bigcup_i (U_i \times U_i)$, we find a section $\hat s$ such that $\Phi_{W + \C \hat s}$ separates $x$ and $y$. 

We only need to do this finitely many times to ensure that $\Phi$ is injective.
\end{proof}

Since the bundle is $K$-invariant, we have a $K$-action on the space of global sections defined by $k s (x) := s (k^{-1} x)$.
We then get a $K$-action on $\Gamma (X, L)^*$ via $k \lambda (s) := \lambda( k^{-1} s)$.
The map in Corollary \ref{LineBundleEmbedding} is equivariant if $W$ is $K$-invariant.

We may define a seminorm on the global CR sections by taking a $K$-invariant open subset $U$ of $X$ over which $L$ is trivial and defining $||s||_U := \textrm{sup}_{x \in U} ||s(x)||^2 + \textrm{sup}_{x \in U} || ds(x)||^2$.  Since $X$ is compact, this induces a norm by taking the supremum over finitely many such seminorms.
Then the $K$-action defined above is continuous. 
From \cite[Proposition 3.6]{Borel}, we conclude that the $K$-finite sections are dense and the embedding from above can be choosen to be equivariant.

\bibliographystyle{alpha}

\bibliography{paper}

\end{document}